%% file: joint_recon.tex
\documentclass{article}

\usepackage{geometry}
\usepackage{fancyhdr}
\usepackage{amsmath,amssymb,amsthm}

\usepackage{algorithm}
\usepackage{algpseudocode}

\usepackage{graphicx}

\usepackage{array,multirow}
\newcolumntype{B}[1]{>{\centering\arraybackslash}m{#1}}
\usepackage{colortbl}

\usepackage[mode=buildnew]{standalone}

\usepackage[percent]{overpic}
\usepackage{color}

\newcommand{\N}{\mathbb{N}}
\newcommand{\R}{\mathbb{R}}
\newcommand{\C}{\mathbb{C}}

\newcommand{\di}{\, \mathrm{d}}
\newcommand{\DM}{\mathrm{D}}
\newcommand{\dx}{\, \mathrm{d}x}
\newcommand{\Leb}{\; \! \mathcal{L}}

\newcommand{\Reg}{\mathcal{R}}

\newcommand{\TV}{\textnormal{TV}}

\newcommand{\W}{\mathcal{W}}
\newcommand{\M}{\mathcal{M}(\Omega,\R^m)}

\newcommand{\CO}{C_0(\Omega,\R^m)}

\newcommand{\Cinf}{C_c^\infty(\Omega,\R^m)}

\newcommand{\BV}{\textnormal{BV}(\Omega)}
\newcommand{\lone}{\ell^1(\R^m)}

\newcommand{\linf}{\ell^{\infty}(\R^m)}

\newcommand{\G}{\mathcal{G}}

\newcommand{\ubold}{\boldsymbol{u}}

\newcommand{\nlone}{\|_{\ell^1(\R^m)}}

\newcommand{\nM}{\|_{\mathcal{M}}}

\newcommand{\B}{\mathcal{B}(\Omega)}

\newcommand{\divergence}{\textnormal{div}}

\newcommand{\D}{\textnormal{D}}

\newcommand{\DKL}{\textnormal{D}_{\textnormal{KL}}}

\newcommand{\prox}{\mathrm{prox}}
\newcommand{\proj}{\mathrm{proj}}
\newcommand{\ICB}{\mathrm{ICB}_{\mathrm{TV}}}
\newcommand{\ICBM}{\mathrm{ICB}_{\| \cdot \nM}}
\newcommand{\Real}{\mathrm{Re}}

\DeclareMathOperator*{\argmin}{arg\,min}

\newtheorem{mydef}{Definition}
\newtheorem{mythm}{Theorem}
\newtheorem{myprop}{Proposition}

\theoremstyle{definition}
\newtheorem{myrem}{Remark}

\numberwithin{equation}{section}

\begin{document}
\title{Joint Reconstruction via Coupled Bregman Iterations \\ with Applications to PET-MR Imaging}
\author{Julian Rasch\thanks{Applied Mathematics M\"unster: Institute for Analysis and Computational Mathematics, 
Westf\"alische Wilhelms-Universit\"at (WWU) M\"unster. Einsteinstr. 62, D 48149 M\"unster, Germany. e-mail: julian.rasch@wwu.de}, Eva-Maria Brinkmann\footnotemark[1], Martin Burger\footnotemark[1]}
\maketitle
\begin{abstract}
Joint reconstruction has recently attracted a lot of attention, especially in the field of medical multi-modality imaging such as PET-MRI. 
Most of the developed methods rely on the comparison of image gradients, or more precisely their location, direction and magnitude, to make use of structural similarities between the images. 
A challenge and still an open issue for most of the methods is to handle images in entirely different scales, i.e. different magnitudes of gradients that cannot be dealt with by a global scaling of the data. 
We propose the use of generalized Bregman distances and infimal convolutions thereof with regard to the well-known total variation functional. 
The use of a total variation subgradient respectively the involved vector field rather than an image gradient
naturally excludes the magnitudes of gradients, which in particular solves the scaling behavior. 
Additionally, the presented method features a weighting that allows to control the amount of interaction between channels. 
We give insights into the general behavior of the method, before we further tailor it to a particular application, namely PET-MRI joint reconstruction. 
To do so, we compute joint reconstruction results from blurry Poisson data for PET and undersampled Fourier data from MRI and show that we can gain a mutual benefit for both modalities.
In particular, the results are superior to the respective separate reconstructions and other joint reconstruction methods. \\

\end{abstract}


\input{introduction}

\input{methods}
\input{numericalsolution}

\input{petmri}

\input{conclusion}

\input{appendix}

\section*{Acknowledgements}

This work has been supported by ERC via Grant EU FP 7 - ERC Consolidator Grant 615216
LifeInverse. MB acknowledges support by the German Science Foundation DFG via SFB 656 Molecular Cardiovascular Imaging, Subproject B2, and EXC
1003 Cells in Motion Cluster of Excellence, M\"unster, Germany.
The authors thank Frank W\"ubbeling (WWU), Florian Knoll (NYU), and Matthias Erhardt (Cambridge) for helpful and stimulating discussions.
Some additional thanks go to the latter for providing easily usable code alongside his work and figuring out the nicest way to design figures for joint reconstruction papers.

\bibliography{references}{}
\bibliographystyle{abbrv}

\end{document}

%% file: introduction.tex
\section{Introduction}

In the last century, the need for better medical diagnostics has been the driving force for developing a large variety of medical imaging systems. 
Accordingly, nowadays it is common practice to perform several experiments on the same patient, e.g. MRI, PET and CT scans, in order to reduce the risk of false interpretation of the measured data. 
To do so, however, one needs reconstruction methods that can exploit the connections between different types of data sets efficiently. 

In general, a wide range of imaging applications requires the solution of an inverse problem 
\begin{align}
Ku = f,
\label{eq:SingleInverseProblem}
\end{align}
where typically $u \colon \Omega \subset \R^d \to \R$ describes a property such as a density of the investigated unknown object. 
A common issue for most of these problems is a general ill-posedness, which makes it difficult to solve for $u$. 
This is mainly due to the properties of the imaging operator $K$ and the possibly poor quality of the available data $f$. 
In order to nevertheless obtain a reasonable solution, it is helpful to gather additional knowledge about $u$. 
This knowledge can be a rather abstract assumption on its properties as a mathematical object, such as being smooth, 
piecewise constant, sparse etc., which is often encoded in the choice of function spaces for the inversion of (\ref{eq:SingleInverseProblem}). 
This process is usually referred to as regularization.
On the contrary one could as well imagine to extend the knowledge about the investigated object by performing several experiments, 
such that one has access to multiple data sets $f_i$, and one has to solve a series of inverse problems for $i=1, \dots, N$,
\begin{align}
K_i u_i = f_i. 
\label{eq:InverseProblem}
\end{align} 
This includes the case of a simple repetition of the same experiment as (\ref{eq:SingleInverseProblem}), 
but as well a change of the experimental setup or the imaging modality, such that one has very different imaging operators $K_i$ and resulting data sets $f_i$.  

Whenever one can assume a relation between the $u_i$, one can try to exploit that connection in order to improve the individual inversions.
The process and the challenge of coupling several inverse problems during the inversion is known as \textit{joint reconstruction} or \textit{joint inversion} 
\cite{Ehrhardt:JointReconstruction,Gallardo:Multiple, Gallardo:Characterization,Gallardo:Joint,Gallardo:Structure,Haber:ModelFusion,Haber:JointInversion}, 
and has attracted a lot of attention in geophysical applications, medical imaging and color image processing. 

The approaches can essentially be divided into two classes: 
either some of the $u_i$ are known a-priori and serve as additional information for the reconstruction of the others in a static mode, or we aim to recover all $u_i$ simultaneously. 
The former is sometimes referred to as model fusion \cite{Haber:ModelFusion} or structural and anatomical priors \cite{Bowsher:Bayesian,Chan:Regularized, Kaipio:InverseProblems, Nuyts:Mutual}. 
It arises e.g in atlas-based reconstructions or attenuation correction in PET-CT, where the attenuation of photons is estimated via a CT image and included in the reconstruction of the PET image \cite{WernickAarsvold}. 
Another application is anatomical MRI priors for quantitative PET imaging \cite{Atre:Evaluation, Leahy1991, Lipinski:Expectation, Vunckx:Evaluation}. 
We will herein focus on a simultaneous reconstruction of all $u_i$, though small adaptions of the proposed approach as well allow for the use of static priors to perform model fusion.

A wide range of recently developed techniques makes use of a similar structure or shape of the images.
Assuming a shared edge set, a common approach is to compare the gradients of the images with regard to position, direction and magnitude. 
For example, Ehrhardt et al. \cite{Ehrhardt:JointReconstruction, Ehrhardt:VectorValuedImageProcessing} derived a measure for the similarity of image gradients, 
called Parallel Level Sets (PLS), using the definition of the Euclidean inner product:
\begin{align}
\int_\Omega |\nabla u||\nabla v| - | \nabla u \cdot \nabla v |\dx = \int_\Omega |\nabla u||\nabla v| (1 - | \cos(\varphi) | )\dx, 
\label{eq:ParallelLevelsets}
\end{align}
where $\varphi$ denotes the angle between $\nabla u$ and $\nabla v$.
A variant thereof is related to a work of Gallardo et al. \cite{Gallardo:Characterization}, using the cross-product of gradients as a measure for structural similarity of images. 
Yet another related approach on coupling of image structures has been proposed by Knoll et al. \cite{Knoll2017} who extended the well-known TV functional \cite{ROF} and its extensions \cite{Bredies2010} to vector-valued data $\ubold = (u_1, \dots, u_n)$, e.g.
\begin{align*}
\TV(\ubold) = \Big\| \| \nabla \ubold \|_M \Big\|_1 = \int_\Omega \|\nabla \ubold(x) \|_M \dx.
\end{align*} 
The important part is then the choice of an appropriate matrix norm $\| \cdot \|_M$ for the point-wise coupling of the gradients. 
A Euclidean norm (or Frobenius norm) leads to a joint TV setting, i.e. 
\begin{align*}
\int_\Omega \|\nabla \ubold(x) \|_2 \dx = \int_\Omega  \left( \sum\nolimits_i |\nabla u_i(x)|^2 \right)^{\frac{1}{2}} \dx, 
\end{align*}
which promotes joint positions of non-zero elements of the gradients. 
Another possible choice is the nuclear norm, i.e. the sum of the singular values of $\nabla \ubold$, which prefers linearly dependent image gradients. \\

All these methods perform well on the positioning and alignment of image gradients. 
However, since they all depend on the magnitude of the gradients, they have difficulties comparing images on different scales, in particular if a global scaling of the data or images is not possible. 
This is e.g. the case if the images share some jumps of equal height, but one image also features small jumps where the others have a large one.
During the reconstruction, this leads to a penalization of the gradient height (and hence image intensities) rather than to an alignment of image gradients. 
This is a significant drawback and still an open issue. 

A particular application where this problem arises is the coupling of Positron Emission Tomography (PET) and Magnetic Resonance Imaging (MRI), since by now both techniques are available in a single device. 
They are able to simultaneously gather functional PET data and structural MR data, which are adjusted in spatial and temporal terms. 
Though inherently different in their physical interpretation, both data sets stem from the same investigated object.
As function is not independent from the underlying anatomy \cite{Ehrhardt:JointReconstruction}, we can hence assume a similar structure in both the PET and MR image. 
However, the images tend to have entirely different scales, and jumps across the edges do not share the same height and sign. 

To overcome that issue, we reformulate a recent work on color image processing \cite{Moeller:ColorBregmanTV} and 
further extend it to a general joint reconstruction approach using the notion of generalized Bregman distances \cite{Burger:BregmanDistancesInInverseProblems}:
\begin{align*}
\D_{J}^p(v,u) = J(v) - J(u) - \langle p, v-u \rangle, \quad p \in \partial J(u). 
\end{align*}
In case of $J = \TV$ the Bregman distance can formally be rewritten as 
\begin{align}
D_{\TV}^p(v,u) = \int_\Omega |\nabla v| \left( 1 - \dfrac{\nabla v}{|\nabla v|} \cdot \dfrac{\nabla u}{|\nabla u|} \right) \dx =  \int_\Omega |\nabla v| \left( 1 - \cos(\varphi) \right) \dx,
\label{eq:AlignmentLevelSets}
\end{align}
with $\varphi$ denoting the angle between $\nabla u$ and $\nabla v$. 
Hence, from a geometric viewpoint, it penalizes the total variation of the first argument $v$, weighted by the alignment of its gradient to the gradient of the second argument $u$. 
In particular, there is no penalty for entirely aligned image gradients ($\cos(\varphi) = 1$), independent of the height of the jump in $v$. 
Equally relevant is the normalization of the gradient of $u$, which {\it excludes} its magnitude from the functional.
This is the key feature that allows for the comparison of images in different scales. 
We point out that the Bregman distance obviously is an asymmetric measure which uses the subgradient as a-priori information. 
We hence follow \cite{Moeller:ColorBregmanTV} and propose the following iterative scheme for a joint reconstruction:
\begin{align}
u_i^{k+1} \in \argmin_{u_i \in \BV} \Bigg\{ \alpha_i H_{f_i}(K_i u_i) + \sum_{j=1}^N w_{ij} D_{\TV}^{p_j^k}(u_i, u_j^k) \Bigg\}, 
\label{eq:ColorBregmanIteration2}
\end{align}
$i = 1, \dots, N$.
Here, the data fidelity terms $H_{f_i}$ enforce a closeness of $u_i$ to the data $f_i$ and the parameters $\alpha_i$ balance between data accuracy and regularization.
In each step of the procedure we solve the latter problem for one $u_i^{k+1}$, using the subgradients of the other $u_j^k$ from the previous iteration as a-priori edge information, thus fitting the edge sets of $u_i^{k+1}$ to the edge sets of all $u_j^k$. 
The matrix $W = (w_{ij}) \in \R^{N\times N}$ serves as a weighting between the different channels and is chosen such that its rows are normalized. 
It allows for different amounts of influence of the different channels and can be adapted to the particular joint reconstruction situation, i.e. according to different data accuracy in different channels, number of channels etc.

However, the iteration scheme \eqref{eq:ColorBregmanIteration2} features a drawback.
A closer look at equation \eqref{eq:AlignmentLevelSets} reveals that it highly penalizes opposing vectors ($\cos(\varphi) = -1$). 
Simply speaking, it favors that jumps across the edges share the same sign. 
In view of different image contrasts, e.g. for medical multi-modality imaging, this is too restrictive.
We therefore again follow \cite{Moeller:ColorBregmanTV} and employ the use of the infimal convolution of Bregman distances, 
to gain a measure favoring linearly dependent image gradients without the edge orientation constraint:
\begin{align*}
\ICB^p(v,u) := [\D_\TV^{p}(\cdot, u) \Box \D_\TV^{-p}(\cdot, -u)](v) = \inf_{v = \phi + \psi} \Big\{ \D_\TV^{p}(\phi, u) + \D_\TV^{-p}(\psi, -u) \Big\}. 
\end{align*}
That way we can find a local decomposition of $v$ such that one part matches the subgradient $p$, and the other part matches the negative subgradient $-p$. 
Without further explanation at this point, we propose the following iteration scheme for a joint reconstruction:
\begin{align}
u_i^{k+1} \in \argmin_{u_i \in \BV} &\Bigg\{ \alpha_i H_{f_i}(K_i u_i) +  w_{ii} D_{\TV}^{p_i^k}(u_i, u_i^k) + \sum_{\substack{j=1 \\ j \neq i}}^N w_{ij} \ICB^{p_j^k}(u_i,u_j^k) \Bigg\}. 
\label{eq:JointReconstructionScheme}
\end{align}
Similarly to \eqref{eq:ColorBregmanIteration2} we enforce a common edge set of the $u_i$ during the iterations, now excluding the orientation of edges.
We point out the use of the Bregman distance on the own channel (or the diagonal), since the edge orientations should matter here.\\ 

The contributions of this paper are twofold: 
on the theoretical side we analyze the structure and behavior of Bregman distances and line out their applicability to a structural joint reconstruction.
We start with the case of infinite $\ell^1$-sequences and the $\ell^1$-norm, and generalize the results to a continuous setting using the notion of finite Radon measures. 
In particular we compute a closed-form representation of the infimal convolution of two Bregman distances with respect to the total variation of a Radon measure and subgradients with opposing sign.
These results are used to finally investigate Bregman distances of distributional derivatives of functions of bounded variation.
On the applied side we show how to solve the derived method numerically, and apply it to the recent problem of PET-MRI joint reconstruction. 
We line out the necessary mathematical modeling of the joint reconstruction problem, reformulate our method in this context and show numerical results on artificial PET-MR data.
In particular we compare with two recent existing joint reconstruction methods, namely Parallel Level Sets \cite{Ehrhardt:JointReconstruction} and Joint Total Variation \cite{Haber:ModelFusion}. 

The remainder of the paper is organized as follows: 
In Section \ref{sec:methods} we derive the method in detail from a geometric viewpoint and describe its behavior analytically, before we show how to numerically solve it in Section \ref{sec:numerics}.
In Section \ref{sec:petmri} we then outline the idea and modeling of PET-MRI joint reconstruction and present our method tailored to that particular application.  
We continue with numerical results for PET-MRI joint reconstruction and then give a short conclusion and outlook in Section \ref{sec:conclusion}.

%% file: methods.tex
\section{Joint reconstruction via Bregman distances}\label{sec:methods}
In this chapter we investigate the use of Bregman distances and infimal convolutions thereof as methods for structural joint reconstruction. 
We start in a discrete setting with the definition of angles between vectors in a Euclidean space and show their connection to a Bregman distance with respect to the $\ell^1$-norm.  
Subsequently, we generalize the idea to the continuous setting using the notion of finite Radon measures, in order to finally compare the (distributional) gradients of functions of bounded variation. 
Finally, we further motivate the use of the infimal convolution in order to eliminate the occurring edge orientation constraint. 

\subsection{The discrete case: vector sequences}
We aim for a regularizer being able to link the edges or level sets of several real-valued functions, but restrict ourselves to two in the following. 
Simply speaking, the gradient of a function is perpendicular to its level sets. 
Hence we want to compare the geometric features of gradients as a vector, i.e. their location, direction and magnitude, and align them in order to obtain a similar structure in the different image channels. 
For the discrete setting let us introduce the space of summable respectively bounded $\R^m$-valued sequences:
\begin{mydef}
 For the Euclidean norm $| \cdot |$ on $\R^m$ and $m \in \N$ we define 
 \begin{alignat*}{3}
 \lone &= \{(\eta_i)_{i \in \N}, \eta_i \in \R^m && : \sum\nolimits_{i=1}^\infty |\eta_i| < \infty && \},  \\
 \linf &= \{(q_i)_{i \in \N}, q_i \in \R^m && : \ \  \sup_{i \in \N} |q_i| < \infty && \}.
\end{alignat*}
\end{mydef}
\noindent
Note that $\linf$ is the dual of $\lone$ under the pairing 
\begin{align*}
 \langle q, \eta \rangle := \sum_{i=1}^\infty q_i \cdot \eta_i.
\end{align*}
So let $\nu,\eta \in \lone$ be two summable sequences of vectors, and let $\varphi_i$ denote the angle between $\nu_i$ and $\eta_i$ ($\varphi_i := 0$ if one of the vectors is zero). 
Then we know by the geometric definition of the Euclidean inner product that the following relation holds true: 
\begin{align}
0 = \cos(\varphi_i) |\nu_i||\eta_i| - \nu_i \cdot \eta_i \leq |\nu_i||\eta_i| - \nu_i \cdot \eta_i, 
\label{eq:EuclideanDotProd}
\end{align}
with equality if and only if $\nu_i$ and $\eta_i$ are parallel ($\cos(\varphi_i)=1$), and maximum value on the right hand side if and only if $\nu_i$ and $\eta_i$ are anti-parallel ($\cos(\varphi_i) = -1$). 
Hence we can consider the right-hand side of \eqref{eq:EuclideanDotProd} as a symmetric measure for the alignment of vectors which increases with the deviation of the vectors.
We can easily extend it for all $i$ by summing up: 
\begin{align*}
\sum_{i=1}^\infty |\nu_i||\eta_i| - \nu_i \cdot \eta_i.
\end{align*}
However, its use for measuring parallelism of vectors imposes some issues. 
Since it measures the magnitude of both vectors as well, we are facing scaling issues for vectors of different magnitudes (or image gradients for images on different scales). 
Additionally, the measure vanishes if one of the vectors is zero, which results in no coupling if one of the images is flat. 
Note that technically this is an advantage for joint reconstruction, since it avoids the artificial transfer of non-shared structures between the images. 
However, in presence of degraded and noisy data the vanishing of either of the vectors would result in no regularization for the other one, which is not desirable.

We shall instead have a look at specific subgradients rather than gradients to overcome the above issues. 
Recall that the subdifferential of the (isotropic) $\| \cdot \nlone$-norm at $\nu$ is given by 
\begin{align}
\partial \|\nu\nlone = \left\{ q \in \linf :\; |q_i|\leq 1,\, q_i = \dfrac{\nu_i}{|\nu_i|} \text{ if } \nu_i \neq 0 \right\}.
\label{eq:l1_subdifferential}
\end{align}
The associated Bregman distance between $\eta$ and $\nu$ and $q \in \partial \|\nu\nlone$ reads
\begin{align*}
 D_{\| \cdot \nlone}^q (\eta, \nu) = \|\eta\nlone - \langle q, \eta \rangle = \sum_{i=1}^\infty |\eta_i| - q_i \cdot \eta_i =  \sum_{i=1}^\infty |\eta_i| (1 - \cos(\varphi_i) |q_i|).
\end{align*}
We can distinguish two situations. 
If $\nu_i \neq 0$, then $q_i = \nu_i/|\nu_i|$, hence the Bregman distance penalizes the magnitude of $\eta_i$, weighted by its deviation from $\nu_i$ in terms of the included angle $\varphi_i$. 
In particular, there is no penalty on $|\eta_i|$ if $\eta_i$ and $\nu_i$ are aligned. 
We as well remark that the magnitude of $\nu_i$ is excluded from the functional, which allows to compare vectors of different magnitude. 
If $\nu_i = 0$, the associated subgradient is an element of the unit ball. 
In order to minimize, we can again distinguish two cases: 
For a rather small magnitude of $q_i$, the benefit of aligning $q_i$ and $\eta_i$ is rather small in comparison to a small $|\eta_i|$, so we are more likely to simply penalize $|\eta_i|$. 
The larger $q_i$ gets, the more beneficial it is to actually align $q_i$ and $\eta_i$.

At first sight, this concept does not seem to be intuitive, since in theory $q$ can be chosen arbitrarily from the unit ball, and there is no reason to align $\eta_i$ to an arbitrary direction. 
However, in practice we do not choose a random $q$ but instead the procedure itself yields a specific subgradient. 
For standard Bregman iterations \cite{Burger:BregmanDistancesInInverseProblems, Osher:AnIterativeRegularizationMethod}, the subgradient is chosen in relation to the data residual (cf. Equation (\ref{eq:ColorBregmanSubgradients})), thus containing \textit{more} than only information about the last iterate. 
Accordingly, as already indicated in \cite{Moeller:ColorBregmanTV}, the subgradient as well serves as an indicator for structures which are likely to appear in the next iterate. 
Hence, thinking of the vectors as image gradients, the length of $q$ can be interpreted as the possibility of an edge at that location, while its direction implies the direction of the jump across that edge. 

\subsection{The continuous case: $\R^m$-valued Radon measures}
In this section we transfer the idea introduced above to a continuous setting. 
Since for a joint reconstruction we want to compare the gradients of functions, natural solution spaces are the Sobolev spaces of weakly differentiable functions, e.g. $W^{1,1}(\Omega)$.
However, these functions do not permit discontinuities. 
In order to allow for jump discontinuities, and hence sharp edges in the reconstructions, the canonical choice is the space of functions of bounded variation. 
Since the (now distributional) derivative of such a function defines a finite Radon measure, the natural extension of the above concept turns out to be the space of finite Radon measures,
for which we shall find an analogous result to the last section. 
Let $d,m \in \N$ and let $\Omega \subset \R^d$ be a bounded and open set.
\begin{mydef}
 Let $\B$ be the Borel $\sigma$-algebra generated by the open sets in $\Omega$. 
 A mapping $\nu \colon \B \to \R^m$ is called an $\R^m$-valued, finite Radon measure on $\Omega$, if it is $\sigma$-additive and $\nu(\emptyset) = 0$. 
 For every set $E \in \B$ we define the total variation $|\nu| \colon \B \to [0,\infty)$ of $\nu$ as 
 \begin{align*}
  |\nu|(E) := \sup \left\{ \sum_{i=1}^{\infty} |\nu(E_i)| ~ :~ (E_i)_{i \in \N} \subset \B \text{ pairwise disjoint, } E = \bigcup_{i=1}^{\infty} E_i \right\}.
 \end{align*}
 Furthermore we denote the space of all finite $\R^m$-valued Radon measures on $\Omega$ by $\M$.
\end{mydef}
We shall need the characterization of $\M$ as a dual space, which can be found in \cite[Thm. 1.54]{Ambrosio:BoundedVariation}.
\begin{myprop}
 The space $\M$ equipped with $\| \nu \nM := |\nu|(\Omega)$ is a Banach space. 
 It can be considered as the dual of $\CO$ under the pairing 
 \begin{align*}
  \langle \nu, \phi \rangle = \int_{\Omega} \phi \cdot \di \nu := \sum_{i=1}^m \int_{\Omega} \phi_i \di \nu_i, 
 \end{align*}
 where $\nu = (\nu_1, \dots, \nu_m) \in \M$ and $\phi = (\phi_1, \dots, \phi_m) \in \CO$. 
 Furthermore, $\| \cdot \nM$ is the dual norm. 
\end{myprop}
As a direct consequence of the Radon-Nikod\'{y}m theorem, the polar decomposition of a finite $\R^m$-valued Radon measure yields a decomposition into its direction and magnitude (cf. \cite{Ambrosio:BoundedVariation}). 
\begin{myprop}
 Let $\nu \in \M$. 
 Then there exists a function $f_\nu \in L_{|\nu|}^1(\Omega, \R^m)$ such that $|f_\nu| = 1$ $|\nu|$-a.e. and for every $E \in \B$ 
 \begin{align*}
  \nu(E) = \int_E f_\nu \di |\nu|. 
 \end{align*}
 $f_\nu$ is unique $|\nu|$-a.e. and we shall refer to it as the density or direction of $\nu$. 
\end{myprop}
We first compute the subdifferential $\partial \| \cdot \nM \subset \M^*$ of the total variation.
Since the dual space $\M^*$ possesses a rather difficult structure (cf. \cite{kaplan:second_dual}), 
a thorough study of the subdifferential goes beyond the scope and intention of this paper. 
We shall thus consider elements in $\mathcal{C}(\nu) := \partial \| \nu \nM \cap \CO$, 
i.e. subgradients that can be identified with a continuous function via the canonical embedding of the predual space $\CO$ into $\M^*$.
It is clear that $\mathcal{C}(\nu)$ can be empty for arbitrary measures $\nu \in \M$.
Hence, for the remainder of this paper we restrict ourselves to measures $\nu$ such that $\mathcal{C}(\nu) \neq \emptyset$, 
which immediately implies that the associated density function $f_\nu$ is continuous $|\nu|$-a.e. 
This can be seen directly from the next theorem.
\begin{myrem}\label{rem:subgrad_cont1}
 We remark that without the assumption $\mathcal{C}(\nu) \neq \emptyset$ very little can be said about the dual pairing between a measure and an element from its dual space, and hence about the structure of the associated Bregman distance.
 The assumption mainly serves illustrative purposes, since it allows to easily interpret the behavior of the method we derive in the course of this section.
 The final method however does {\it not} depend on the existence of a continuous subgradient (see also Remark \ref{rem:subgrad_cont}).
\end{myrem}
\begin{mythm}\label{thm:subdiff}
 Let $\nu \in \M$ with polar decomposition $\nu = f_\nu |\nu|$ such that $\mathcal{C}(\nu) := \partial \| \nu \nM \cap \CO \neq \emptyset$. 
 Then $q \in \mathcal{C}(\nu)$ if and only if $q \in \CO$ with $\| q \|_\infty \leq 1$ and  $q = f_\nu ~ |\nu|$-a.e.
\end{mythm}
\begin{proof}
 First, let $q \in \mathcal{C}(\nu)$. 
 Then by the definition of the subdifferential we have 
 \begin{align}
  \langle q, \eta - \nu \rangle \leq \| \eta \nM - \| \nu \nM 
  \label{eq:subdiff_Mnorm}
 \end{align}
 for all $\eta \in \M$. 
 In particular, for arbitrary $x \in \Omega$ let the Dirac measure $\delta_x \colon \B \to \R$ be defined by
 \begin{align*}
  \delta_x (E) = \begin{cases}
                  1, & x \in E \\
                  0, & \text{else}
                 \end{cases}
 \end{align*}
 for all $E \in \B$.
 Then for any $ \varepsilon >0 $ and fixed $\sigma \in S^{m-1} = \{ \sigma \in R^m ~|~ |\sigma| = 1\}$, $\varepsilon \sigma \delta_x \in \M$ defines a finite Radon measure.
 Hence for $\nu_\varepsilon := \nu + \varepsilon \sigma \delta_x$ we find $\| \nu_\varepsilon \nM \leq \| \nu \nM + \varepsilon$ and 
  \begin{align*}
\langle q, \nu_\varepsilon - \nu \rangle \leq \|\nu_\varepsilon \nM - \| \nu \nM 
 \quad \Rightarrow \quad \langle q, \sigma \delta_x \rangle \leq 1.
 \end{align*}
 Since $\sigma$ was arbitrary we deduce 
 \begin{align*}
  \sup_{\sigma \in S^{m-1}} | \langle q, \sigma \delta_x \rangle | 
  = \sup_{\sigma \in S^{m-1}} \left| \int_\Omega q \cdot \sigma \di \delta_x \right| 
  = \sup_{\sigma \in S^{m-1}} | (q \cdot \sigma)(x) |  \leq 1
 \end{align*}
 which yields $|q(x)| \leq 1$ for the particular choice $\sigma = \frac{q}{|q|}(x)$.
 Since $x$ was arbitrary we deduce that $\|q\|_\infty \leq 1$.
 Furthermore using $\eta = 0$ and $\eta = 2\nu$ in \eqref{eq:subdiff_Mnorm} we obtain 
 \begin{align*}
  \quad \langle q,\nu \rangle  = \| \nu \nM 
  \quad \Leftrightarrow \quad \int_\Omega q \cdot \di \nu = \int_\Omega 1 \di |\nu| 
  \quad \Leftrightarrow \quad \int_\Omega q \cdot f_\nu \di |\nu| = \int_\Omega 1 \di |\nu|, 
 \end{align*}
 hence $ 1 = q \cdot f_\nu$ $|\nu|$-a.e. (note that $1 - q \cdot f_\nu \geq 0$ $|\nu|$-a.e.). 
 Since $|f_\nu| = 1$, this implies $q = f_\nu$ $|\nu|$-a.e. and in particular that $f_\nu$ is continuous $|\nu|$-a.e.
 
 On the other hand, for any $q \in \CO$ with $\|q\|_\infty \leq 1$ and $q = f_\nu$ $|\nu|$-a.e. we calculate for any $\eta \in \M$: 
 \begin{alignat*}{3}
  \langle q, \eta - \nu \rangle 
  &= \int_\Omega q \cdot f_\eta \di |\eta|  - \int_\Omega q \cdot f_\nu \di |\nu| 
  &&\leq \int_\Omega |q| |f_\eta| \di |\eta| - \int_\Omega f_\nu \cdot f_\nu \di |\nu| \\
  & \leq \|q\|_\infty \int_\Omega 1 \di |\eta| - \int_\Omega 1 \di |\nu| 
  &&\leq \|\eta \nM - \| \nu \nM,
 \end{alignat*}
 which yields the assertion.
\end{proof}
Note that the structure of the $\| \cdot \nM$-subdifferential is similar to the structure of the $\ell^1$-subdifferential from the last subsection. 
A subgradient $q \in \mathcal{C}(\nu)$ equals the direction $f_\nu$ of the measure $\nu$ on its support, while it is an arbitrary element from the unit ball on the zero sets of $|\nu|$.

Let now $\eta, \nu \in \M$ be two $\R^m$-valued Radon measures with polar decomposition $\eta = f_\eta |\eta|, \nu = f_\nu |\nu|$ such that $\mathcal{C}(\nu) \neq \emptyset$. 
We can immediately write down the Bregman distance between $\eta$ and $\nu$ with respect to $q \in \mathcal{C}(\nu)$ to obtain an analogous result to the discrete case. 
\begin{align*}
 D_{\| \cdot \nM}^q (\eta, \nu) = \| \eta \nM - \langle q, \eta \rangle = \int_\Omega 1 \di |\eta| - \int_\Omega q \cdot f_\eta \di |\eta| = \int_\Omega 1 - q \cdot f_\eta \di |\eta|.
\end{align*}
We again denote the angle between $q$ and $f_\eta$ by $\varphi$,
\begin{align}
 D_{\| \cdot \nM}^q (\eta, \nu) = \int_\Omega 1 - \cos(\varphi)|q| \di |\eta|,
 \label{eq:BDM}
\end{align}
and can again distinguish two situations to minimize \eqref{eq:BDM}, now in an $|\eta|$-a.e. sense.
In case $q = f_\nu$, i.e. on the support of $\nu$, it is beneficial to either align $f_\eta$ to $f_\nu$ or to assign no mass to $\eta$.
If $|q| < 1$, i.e. outside the support of $\nu$, it does not provide any explicit side information. 
Instead, depending on the magnitude of $q$ it is either possible to assign no mass to $\eta$, or to align its direction $f_\eta$ to $q$. 

We shall finally transfer the concept to the (distributional) gradients of $\BV$-functions in order to compare the structures of images. 
\begin{mydef}
 A function $u \in L^1(\Omega)$ is a function of bounded variation in $\Omega$ if and only if its distributional derivative is representable by a finite Radon measure $\DM u = (\DM_1 u, \dots, \DM_m u)$ in $\Omega$, i.e. 
 \begin{align*}
  \int_{\Omega} u \, \divergence(\phi) \di x = - \sum_{i=1}^m \int_\Omega \phi_i \di \DM_i u \quad \text{ for all } \quad \phi \in \Cinf.
 \end{align*}
 We denote the vector space of all functions of bounded variation by $\BV$.
 An equivalent formulation (see \cite[Prop. 3.6.]{Ambrosio:BoundedVariation}) is that $u \in \BV$ if and only if 
 \begin{align*}
  \TV(u) := \sup_{\substack{\phi \in \Cinf, \\ \|\phi\|_{\infty} \leq 1}} \int_\Omega u \, \divergence(\phi) \di x < \infty.
 \end{align*}
 In particular, $\TV(u) = \| \DM u \nM$.
\end{mydef}
%
Hence let $\DM \colon \BV \to \M$ denote the distributional derivative.
We obtain the following result. 
\begin{myprop}
Let $u,v \in \BV$ with derivatives $\DM u, \DM v \in \M$, and $p \in \partial \TV(u)$. 
Then
\begin{align*}
 D_{\TV}^p(v,u) = D_{\| \cdot \nM}^q (\DM v, \DM u) 
\end{align*}
for some $q \in \partial \| \DM u \nM$.
\label{prop:BregmanDistance}
\end{myprop}
\begin{proof}
Since $\DM$ is continuous between $\BV$ and $\M$ and the underlying functionals are convex and lower semicontinuous, the chain rule \cite[p.27]{Ekeland:ConvexAnalysis} yields 
$p \in \partial \TV (u)$ if and only if $p = \DM^* q$ for some $q \in \partial \| \DM u \nM$.
Thus 
\begin{align*}
 D_{\TV}^p(v,u) = \TV(v) - \langle p,v \rangle = \| \DM v \nM - \langle \DM^* q , v \rangle = \| \DM v \nM - \langle q , \DM v \rangle =  D_{\| \cdot \nM}^q (\DM v, \DM u).
\end{align*}
\end{proof}
Let now $u,v \in \BV$ such that $\mathcal{C}(\D u) \neq \emptyset$.
From 
\begin{align}
 D_{\| \cdot \nM}^q (\DM v, \DM u) = \int_\Omega 1 - q \cdot f_{\DM v} \di |\DM v|, \quad q \in \partial \| \DM u \nM
 \label{eq:BDTV}
\end{align}
and our above observations we see, that the Bregman distance with respect to the TV functional penalizes the total variation of $v$ weighted by the deviation of image gradients (respectively their direction). 
Minimizing \eqref{eq:BDTV} hence favors a piecewise constant (denoised) image $v$ with a similar edge set (and hence structure) as $u$.  
Note that we again obtain no penalty for aligned gradient directions, and that it is in particular possible {\it not} to introduce an edge to $v$ even though $\DM u$ indicates it. 
This feature reduces the risk of artificially introducing edges present in one image while the other is flat.

We give a brief example to further illustrate the idea. 
The distributional derivative of $u$ from the Sobolev space $W^{1,1}(\Omega) \subset \BV$ is given by $\DM u = \nabla u \Leb$ with the Lebesgue measure $\Leb$. 
The associated polar decomposition is given by $f_{\DM u} |\nabla u| \Leb$ with $f_{\DM u} = \nabla u / |\nabla u|$ on the support of $\nabla u$.
If $q \in \mathcal{C}(\D u)$ then $q = \nabla u / |\nabla u|$, $|\nabla u| \Leb$-a.e., and the Bregman distance between $u$ and $v \in W^{1,1}(\Omega)$ reads 
\begin{align*}
 D_{\TV}^p(v,u) = \int_\Omega 1 - q \cdot \frac{\nabla v}{| \nabla v |} \di |\nabla v| \Leb.
\end{align*}
Minimizing the Bregman distance hence corresponds to weighted total variation, and favors an alignment of $\nabla v$ to $\nabla u$, $|\nabla u| \Leb$-a.e. 

Since the Bregman distance is not symmetric, but the subgradients serve as one-sided a-priori information, we propose to solve the following iterative procedure for structural joint reconstruction: 
\begin{align}
u_i^{k+1} \in \argmin_{u_i \in \BV} \Bigg\{ \alpha_i H_{f_i}(K_i u_i) + \sum_{j=1}^N w_{ij} D_{\TV}^{p_j^k}(u_i, u_j^k) \Bigg\}. 
\label{eq:ColorBregmanIteration}
\end{align}
In each step of the procedure we solve the latter problem for one $u_i$, using the subgradients of the other $u_j$ as a-priori information, thus fitting the level sets of $u_i$ to the level sets of all $u_j$. 
Note that we have included the last iterate of $u_i$ as well, which guarantees the well-definedness of the problem and the usual properties of single channel Bregman iterations \cite{Burger:BregmanDistancesInInverseProblems, Osher:AnIterativeRegularizationMethod}. 
The matrix $W = (w_{ij}) \in \R^{N\times N}$ serves as a weighting between the different channels and is chosen such that its rows are normalized, i.e. 
\begin{align}
\sum_{j=1}^N w_{ij} = 1, \quad \forall \; i = 1, \dots, N.
\label{eq:RowSum}
\end{align}
Note that in case of $W$ being the identity we obtain decoupled standard Bregman iterations. 
The iteration scheme \eqref{eq:ColorBregmanIteration} is of course not limited to the total variation but can be carried out for arbitrary convex and absolutely one-homogeneous regularizers $J$. 
It has been proposed before by Moeller et al. \cite{Moeller:ColorBregmanTV} for color image processing and 
\begin{align*}
H_{f_i}(K_i u_i) = \dfrac{1}{2} \|K_i u_i - f_i\|^2.
\end{align*} 
Starting with $u^0 = 0$ and $p^0 = 0$, the optimality condition of problem \eqref{eq:ColorBregmanIteration} allows us to compute new subgradients after having updated every $u_i$ once: 
\begin{align}
p_i^{k+1} = \sum_{j=1}^N w_{ij} p_j^k - \dfrac{1}{\alpha_i} K_i^* \nabla H_{f_i} (K_i u_i^{k+1}), \quad p_i^{k+1} \in \partial \TV(u_i^{k+1}).
\label{eq:ColorBregmanSubgradients}
\end{align} 

\begin{myrem}\label{rem:subgrad_cont}
 We mention that the set $\mathcal{C} (\DM u_i^{k+1})$ can be empty for the solution $u_i^{k+1}$ of method \eqref{eq:ColorBregmanIteration}, i.e. that there does not exist a {\it continuous} $q_i^{k+1} \in \partial \| \DM u_i^{k+1} \nM$ such that $\DM^* q_i^{k+1} = p_i^{k+1}$.
 It is important to notice that the above theory and Proposition \ref{prop:BregmanDistance} still hold in this situation, with the exception of the {\it illustration} \eqref{eq:BDTV}.
 The derived method \eqref{eq:ColorBregmanIteration} solely depends on the $p_i^{k+1}$ from \eqref{eq:ColorBregmanSubgradients} and is hence well-defined also for $\mathcal{C} (\DM u_i^{k+1}) = \emptyset$.
 See also Remark \ref{rem:subgrad_cont1}.
\end{myrem}

With the subgradient updates \eqref{eq:ColorBregmanSubgradients} we shall briefly prove the well-definedness of the iteration scheme \eqref{eq:ColorBregmanIteration} by the direct method of the calculus of variations \cite[Prop. 1.2]{Ekeland:ConvexAnalysis}. 
Therefore, let $\W$ be the predual space of $\BV$, and assume that $K_i$ is the adjoint of a linear and bounded functional $L_i \colon Y_i^* \to \W$ on a Hilbert space $Y_i$ for every $i$.

%
\begin{mythm}
 Let the mapping $u_i \mapsto \mathcal{E}_i(u_i) = \alpha_i H_{f_i}(K_i u_i) + \TV(u_i)$ be coercive and weak-$^*$ lower semicontinuous and $p_j^k$ be given by \eqref{eq:ColorBregmanSubgradients} for $k \geq 0$.
 Assume furthermore that the convex conjugate $H_{f_i}^*$ of $H_{f_i}$ is uniformly bounded and $w_{ii} > 0$. 
 Then there exists at least one solution to \eqref{eq:ColorBregmanIteration}.
 \label{thm:existence}
\end{mythm}
\begin{proof}
For simplicity we let $\alpha_i = 1$, and put iteration numbers in brackets in order to distinguish between powers and iteration numbers in this proof.
We inductively show that  
\begin{align*}
p_i^{(k)} = \sum_{l=0}^{k-1} \sum_{j \neq i} w_{ij} w_{ii}^{l-1} p_j^{(k-l-1)} - L_i \zeta
\end{align*}
for some $\zeta \in Y_i^*$. 
The crucial part is that due to \eqref{eq:RowSum}
\begin{align*}
\sum_{l=0}^{k-1} \sum_{j \neq i} w_{ij} w_{ii}^{l-1} = \sum_{j \neq i} w_{ij} \dfrac{1-w_{ii}^k}{1-w_{ii}} = 1-w_{ii}^k < 1 \quad \forall k.
\end{align*}
By the positivity of Bregman distances we have for $0 < \delta < 1$
\begin{alignat*}{2}
\mathcal{E} (u_i)  :=
 & H_{f_i}(K_i u_i) +  \sum_{j=1}^N w_{ij} D_\TV^{p_j^{(k)}}(u_i, u_j^{(k)}) \\
 \geq  &~H_{f_i}(K_i u_i) + \delta w_{ii} \big( \TV(u_i) - \langle p_i^{(k)} , u_i \rangle \big) \\
= &~H_{f_i}(K_i u_i) + \delta w_{ii} \big( \TV(u_i) - \langle \sum_{l=0}^{k-1} \sum_{j \neq i} w_{ij} w_{ii}^{l-1} p_j^{(k-l-1)}, u_i \rangle + \langle L_i \zeta, u_i \rangle \big) \\
= &~H_{f_i}(K_i u_i) + \delta w_{ii}^{k+1} \TV(u_i) + \delta w_{ii} \langle L_i \zeta, u_i \rangle +  \delta w_{ii} \big( \sum_{l=0}^{k-1} \sum_{j \neq i}  w_{ij} w_{ii}^{l-1} \big( \underbrace{\TV(u_i) - \langle p_j^{(k-l-1)}, u_i \rangle}_{\geq 0} \big) \big) \\
\geq &~(1-\delta) H_{f_i}(K_i u_i) + \delta w_{ii}^{k+1} \TV(u_i) + \delta \big( H_{f_i}(K_i u_i) - \langle  - w_{ii} \zeta, K_i u_i \rangle \big)\\
\geq &~c_1 ( H_{f_i}(K_i u_i) + \TV(u_i) ) - \delta H_{f_i}^*(-w_{ii} \zeta) \\
\geq &~c_1 \mathcal{E}_i(u_i) - c_2, 
\end{alignat*}
where $c_1 = \min \{ 1- \delta, \delta w_{ii}^{k+1} \}$. 
Hence the mapping $u_i \mapsto \mathcal{E}_i(u_i)$ is bounded on every sublevel set of the objective functional, and the Banach-Alaoglu theorem 
implies their precompactness in the weak-$^*$ topology. 
Since $p_j^k \in \W$ for all $j$, the mapping $u_i \mapsto \mathcal{E}(u_i)$ is weak-$^*$ sequentially lower semicontinuous on bounded sets, which yields the assertion by a standard argument (see e.g. \cite[Prop. 1.2]{Ekeland:ConvexAnalysis}).
\end{proof}
\begin{myrem}
 We like to further comment on the assumption $K_i = L_i^*$. 
 Following the line of argumentation in \cite{Bredies2013}, this assumption ensures that $K_i^*$ indeed maps $Y_i$ into $\W$ and not into the bigger space $\BV$. 
 More precisely, this assumption implies that $K_i$ is sequentially continuous from the weak-$^*$ topology of $\BV$ to the weak(-$^*$) topology of $Y_i$ (note that they coincide on a Hilbert space) and hence possesses an adjoint $K_i^*$ mapping $Y_i$ to $\W$, regarded as a closed subspace of $\BV$.
 As a consequence, $p_i^{k+1} \in \W$, which is necessary to ensure lower semicontinuity in the weak-$^*$ topology of  the linear term appearing in the Bregman distance.
\end{myrem}

\begin{myrem}\label{rem:convergenceCBTV}
 It is natural to ask the question about the convergence of the entire iterative procedure \eqref{eq:ColorBregmanIteration}. 
 In \cite[Sec. 4.4 and 5]{Moeller:ColorBregmanTV} the authors discuss potential stationary solutions of the procedure and its convergence in some special cases. 
 In case of identical operators $K_i = K$, symmetric weights $w_{ij} = w_{ji}$ and squared Hilbert space norms as data fidelities one can show that \eqref{eq:ColorBregmanIteration} converges for noiseless data $f_i$. 
 With noisy data it is shown that the Bregman distance to the clean images decreases until the residual reaches the noise level, which is a generalization of the results for single channel Bregman iterations \cite{Osher:AnIterativeRegularizationMethod}.
 Unfortunately, little can be said about the general case, which at this point is subject of future research. 
 The numerical studies however are encouraging, as they basically show the same behavior of \eqref{eq:ColorBregmanIteration} as single channel Bregman iterations even for very different imaging operators $K_i$ and asymmetric weightings. 
\end{myrem}

\subsection{Allowing anti-parallel vectors: Infimal convolution of Bregman distances}
Taking a second glance at \eqref{eq:BDM} or \eqref{eq:BDTV}, we notice that the derived measure highly penalizes opposing vectors ($\cos(\varphi) = -1$) and forces them to point to the same direction, which requires that jumps across the edges of several images share the same sign. 
Or simply speaking, taking a ``step up'' across the edge between two structures in one image requires a ``step up'' in the other images, too.
This can be suitable for some applications, e.g. for color image denoising and processing \cite{Moeller:ColorBregmanTV}, since edges in different color channels of natural images most commonly seem to point to the same direction. 
However, this assumption is not always valid and in particular questionable in PET-MRI, so we adjust our method to allow opposing jumps across the edges. 
In view of \eqref{eq:BDM} the most intuitive way is to apply the absolute value to the dot-product, i.e. 
\begin{align}
\int_\Omega 1 - |q \cdot f_\eta| \di |\eta| = \int_\Omega 1 - |\cos(\varphi)||q| \di |\eta|
\label{eq:BDMabs}
\end{align}
Note that this resembles Ehrhardt et al.'s idea of parallel level sets \eqref{eq:ParallelLevelsets} but bears the problem of non-convexity, and is hence not related to a Bregman setting. 
Instead we follow \cite{Moeller:ColorBregmanTV} and use the infimal convolution of two Bregman distances. 
The infimal convolution of two proper and convex functionals $J,L$ is defined as \cite[p. 167]{Bauschke:ConvexAnalysis}
\begin{align*}
(J \Box L)(u) = \inf_{\phi + \psi = u} J(\phi) + L(\psi).
\end{align*}
It is easy to show that the infimal convolution preserves convexity. 
It allows for a local decomposition of the argument $u$ into two parts, each trying to minimize one of the underlying functionals, while balancing the minimization of both. 
With the aim of losing the sign condition, it makes sense not only to look at the subgradient $q$, but also at the subgradient $-q$ with negative or ``inverted'' contrast. 
We recall that for absolutely one-homogeneous functionals $J$ we have 
\begin{align*}
q \in \partial J(u) \Leftrightarrow -q \in \partial J(-u),
\end{align*}
so it is natural to look at the following infimal convolution of two $\| \cdot \nM$-Bregman distances. 
\begin{mythm}\label{thm:inf_conv}
Let $\nu,\eta \in \M$ such that $\mathcal{C}(\nu) \neq \emptyset$.  
Let $q \in \mathcal{C}(\nu)$. Then
\begin{align*}
\ICBM^q(\eta,\nu) := \left[D_{\| \cdot \nM}^q(\cdot,\nu) \Box D_{\| \cdot \nM}^{-q}(\cdot,-\nu)\right](\eta) 
= \int_\Omega G(f_\eta,q) \di |\eta| , 
\end{align*}
where 
\begin{align}
G(f_\eta,q) = \begin{cases}
	       1 - |\cos(\varphi)| |q|, & \text{ if } |q| < |\cos(\varphi)|,\\
	       | \sin(\varphi)| \sqrt{1 - |q|^2}, &\text{ if } |q| \geq |\cos(\varphi)|
	       \end{cases}
	       \label{eq:IC_integrand}
\end{align}
$|\eta|$-a.e.,
 and $\varphi$ denotes the angle between $f_\eta$ and $q$, i.e. 
 $\cos(\varphi) |f_\eta| |q| = f_\eta \cdot q$.
\end{mythm}
%
 Since we already proved the point-wise case in \cite{2016debiasing} and the proof follows analogously, we have only included it in the Appendix.
%
At first we recognize that $\ICBM$ behaves like \eqref{eq:BDMabs} on one part of its domain, and that we are in particular losing the sign of $q$ everywhere. 
Furthermore, since the Bregman distance is convex in its first argument if we fix the second, this functional is convex. 
For a detailed description of the behavior, we at first pass over to the gradient setting again, i.e. for $q \in \mathcal{C}(\D u)$
\begin{align}
\ICBM^q( \DM v, \DM u) =  \left[D_{\| \cdot \nM}^q(\cdot,\DM u) \Box D_{\| \cdot \nM}^{-q}(\cdot,-\DM u)\right](\DM v) 
= \int_\Omega G(f_{\DM v},q) \di |\DM v|. 
\label{eq:BDIC}
\end{align}
We again distinguish two different situations: 
either $u$ has an edge, which corresponds to a non-zero (distributional) gradient, or $u$ is constant with vanishing gradient. 
In the former we find that $|q|= |f_{\DM u}| =1$, which is only possible in the ``else''-case of \eqref{eq:IC_integrand}, and the functional vanishes. 
This results in no penalization of $\DM v$, which allows to introduce an edge of arbitrary height to $v$ for free, hence again locating edges at the same positions.
We observe that here the direction of the edge is arbitrary and refer the reader to \cite{2016debiasing} for some further elaboration on that. 
In case it is constant, $u$ does not offer any structural information, but it is desirable to nevertheless have a regularizing effect on $v$ or $\DM v$, respectively. 
However, the interpretation is a bit more delicate in that case.
The subgradient $q$ here is a vector from the unit ball with length $|q|\leq 1$. 
\begin{figure}
   \center
   \begin{tabular}{B{6cm}B{1cm}B{6cm}}
      \includegraphics[width=0.2\textwidth]{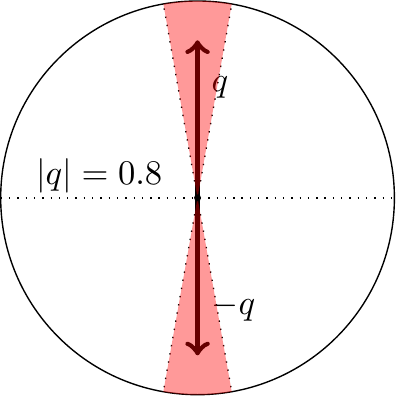} \hfill && 
      \includegraphics[width=0.2\textwidth]{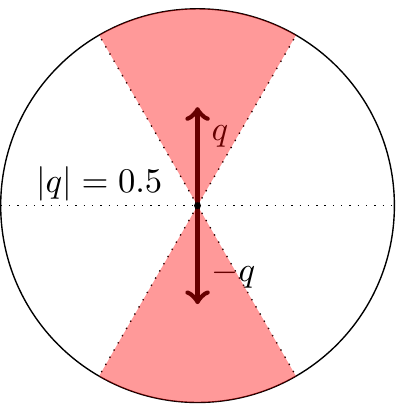} \\
      (a) Large magnitude of $q$, high probability of an edge and small range of possible directions. & &
      (b) Medium magnitude of $q$, medium probability of an edge and wider range of possible directions.
   \end{tabular}
   \caption{Infimal convolution of two $\| \cdot \nM$-Bregman distances. 
 In order to minimize the functional \eqref{eq:BDIC}, it is beneficial to pick $f_\eta$ respectively $f_{\DM v}$ from the colored range, 
 which narrows with increasing $|q|$. 
 Hence $|q|$ serves as a probability of how likely an edge is at that position, while its direction implies a range for the direction of the jump.}
   \label{fig:infconv_angles2}
\end{figure}	
Inserting the conditions into the integrand of $\eqref{eq:IC_integrand}$ yields an upper bound for the value of the integrand in the first case, and a lower bound for the second case. 
For the first case and $|q| < |\cos(\varphi)|$ we have
\begin{align*}
 1 - |\cos ( \varphi )| |q| < 1 - |q|^2.
\end{align*}
For the second case we find 
\begin{align*}
 |q| \geq |\cos(\varphi)| \Leftrightarrow 1 - |q|^2 \leq 1 - |\cos(\varphi)|^2 \Leftrightarrow \sqrt{1 - |q|^2} \leq |\sin(\varphi)|
\end{align*}
and thus for the value of the integrand 
\begin{align*}
 |\sin(\varphi)| \sqrt{1 - |q|^2} \geq 1 - |q|^2. 
\end{align*}
We combine both results and obtain
\begin{align*}
 1 - |\cos ( \varphi )| |q| < |\sin(\varphi)| \sqrt{1 - |q|^2}. 
\end{align*}
So in order to minimize the functional it is beneficial to either assign no mass to $\DM v$ or to find a combination of $f_{\DM v}$ and $q$ such that $|q| < |\cos(\varphi)|$. 
Figure \ref{fig:infconv_angles2} illustrates the situation for different values of $|q|$. 
The larger the magnitude of $q$, the narrower is the range of vectors around $q$ that we can pick to fulfill $|q| < |\cos(\varphi)|$. 
Hence the functional tries to align $f_{\DM v}$ to $q$ to obtain the smallest penalty. 
We again find that $|q|$ can serve as a probability of how likely an edge is at that position, while its direction implies a range for the direction of the jump which narrows with increasing size of $q$.
\begin{myprop}
Let $u,v \in \BV$ with derivatives $\DM u, \DM v \in \M$, and $p \in \partial \TV(u)$.  
Then 
\begin{align*}
\TV(v) - |\langle p,v \rangle | \geq [D_\TV^p(\cdot,u) \Box D_\TV^{-p}(\cdot,-u)](v) \geq [D_{\| \cdot \nM}^q(\cdot, \DM u) \Box D_{\| \cdot \nM}^{-q}(\cdot, -\DM u)](\DM v)
\end{align*} 
respectively
\begin{align*}
\TV(v) - |\langle p,v \rangle | \geq \ICB^p(u,v) \geq \ICBM^q(\DM u, \DM v)
\end{align*}
for some $q \in \partial \| \DM u \nM$.
\end{myprop}

\begin{proof}
For $w=0$ and $w=v$ we have
\begin{align*}
[D_\TV^p(\cdot,u) \Box D_\TV^{-p}(\cdot,-u)](v) &= \inf_{w \in \BV} \TV(v-w) - \langle p,v-w \rangle + \TV(w) + \langle p,w \rangle \\
&\leq \TV(v) - \langle p,v \rangle, \\
[D_\TV^p(\cdot,u) \Box D_\TV^{-p}(\cdot,-u)](v) &= \inf_{w \in \BV} \TV(v-w) - \langle p,v-w \rangle + \TV(w) + \langle p,w \rangle \\
&\leq \TV(v) + \langle p,v \rangle,
\end{align*}
which yields the first inequality. Then by definition we have
\begin{align*}
&\quad \;  [D_\TV^p(\cdot,u) \Box D_\TV^{-p}(\cdot,-u)](v) \\
&= \inf_{w \in \BV}  \TV(v-w) - \langle p,v-w \rangle + \TV(w) + \langle p,w \rangle \\
&= \inf_{w \in \BV}  \| \DM v - \DM w \nM - \langle q, \DM v - \DM w \rangle + \| \DM w \nM + \langle q, \DM w \rangle \\
&= \inf_{z = \DM w \in \M}  \| \DM v - z \nM - \langle q, \DM v - z \rangle + \| z \nM + \langle q, z \rangle  \\
&\geq \inf_{z \in \M}  \| \DM v - z \nM - \langle q, \DM v - z \rangle + \| z \nM + \langle q, z \rangle \\
&= [D_{\| \cdot \nM}^q(\cdot, \DM u) \Box D_{\| \cdot \nM}^{-q}(\cdot, -\DM u)](\DM v).
\end{align*}
\end{proof}
As a consequence, minimizing $\ICB^p(u,v)$ immediately implies a small $\ICBM^q(\DM v, \DM u)$ and hence the desired behavior for joint reconstruction. 
Following \cite{Moeller:ColorBregmanTV}, we propose the following iteration scheme: 
\begin{align}
u_i^{k+1} \in \argmin_{u_i \in \BV} &\Bigg\{ \alpha_i H_{f_i}(K_i u_i) + w_{ii} D_{\TV}^{p_i^k}(u_i, u_i^k) + \sum_{\substack{j=1 \\ j \neq i}}^N w_{ij} \ICB^{p_j^k}(u_i,u_j^k) \Bigg\}. 
\label{eq:JointReconstructionScheme2}
\end{align}
Similarly to \eqref{eq:ColorBregmanIteration}, we solve the latter problem for one $u_i$ in each step, using the subgradients of the other $u_j, j \neq i$ as a-priori information. 
We hence fit the edge set of $u_i^{k+1}$ to the edge sets of all $u_j^k$, now excluding the direction/sign constraint via the infimal convolution. 
We emphasize the use of the usual Bregman distance on the diagonal, which is the obvious choice, since in contrast to all the foreign channels the sign should be relevant on the own channel. 

Starting with no prior knowledge about subgradients, i.e. $p_i^0 = 0$ for all $i = 1, \dots, N$, the first results $u_i^1$ are separate TV-reconstructions without any coupling. 
We then update the subgradients for the first time and gain coupled reconstructions for all subsequent iterations. 
The general procedure is similar to the use of usual Bregman iterations \cite{Osher:AnIterativeRegularizationMethod, Burger:BregmanDistancesInInverseProblems}: 
Starting with a very high regularization ($\alpha_i$ small), the first iterates are very smooth and only contain large scales. 
Hence the subgradients essentially contain information about the edges of large image features. 
Since for Bregman iterations the regularization behaves proportional to $\frac{1}{\alpha_i k}$, the amount of regularization decreases with every iteration and we gain more structure in the subgradients, 
i.e. a better knowledge about edge locations of smaller features. 
The iteration has to be stopped when noise reappears in the images. 
The ratio between the parameters $\alpha_i$ is crucial for the outcome of the joint reconstruction, as we have to ensure that the reconstructions evolve equally fast. 

\begin{myrem}
 Unfortunately, we cannot immediately obtain the necessary subgradient updates from the optimality condition of the problem due to the involved infimal convolution. 
 Without an explicit structure of the subgradient (cf. \eqref{eq:ColorBregmanSubgradients}) it is in particular impossible to prove the well-definedness of the minimization problems in a similar manner as Theorem \ref{thm:existence}, since there is no particular reason for the problem to be coercive for an arbitrary subgradient $p_i$.
 For the same reason it is hard to characterize a potential convergence of procedure \eqref{eq:JointReconstructionScheme2} (cf. Remark \ref{rem:convergenceCBTV}). 
 However, for the numerical realization of the method, we shall find an alternative, numerical update for the subgradients using a primal-dual scheme.
\end{myrem}

We already pointed out that both methods \eqref{eq:ColorBregmanIteration} and \eqref{eq:JointReconstructionScheme2} naturally exclude the size of image gradients.
Interestingly, they additionally show a certain kind of invariance under positive rescaling of the data $f_i$. 
\begin{myprop}
 For any $c >0 $ let $H_{f_i}$ satisfy $H_{c f_i}(c u_i) = c^r H_{f_i}(u_i)$ for some $r \geq 1$. 
 Furthermore, for fixed $\alpha_i$ and $f_i$, denote the solution of \eqref{eq:ColorBregmanIteration} respectively \eqref{eq:JointReconstructionScheme2} by $u_i^{k+1}$. 
 Then the solution of \eqref{eq:ColorBregmanIteration} respectively \eqref{eq:JointReconstructionScheme2} for any rescaling $\tilde{f}_i = c_i f_i$ of the data $f_i$ by $c_i > 0$ and regularization parameter $\tilde{\alpha}_i = \alpha_i / c_i^{r-1}$ is given by $\tilde{u}_i^{k+1} = c_i u_i^{k+1}$.
\end{myprop}
\begin{figure}[t!]
 \center
 \begin{tabular}{ccc}
   \includegraphics[height=5cm, width=0.3\textwidth]{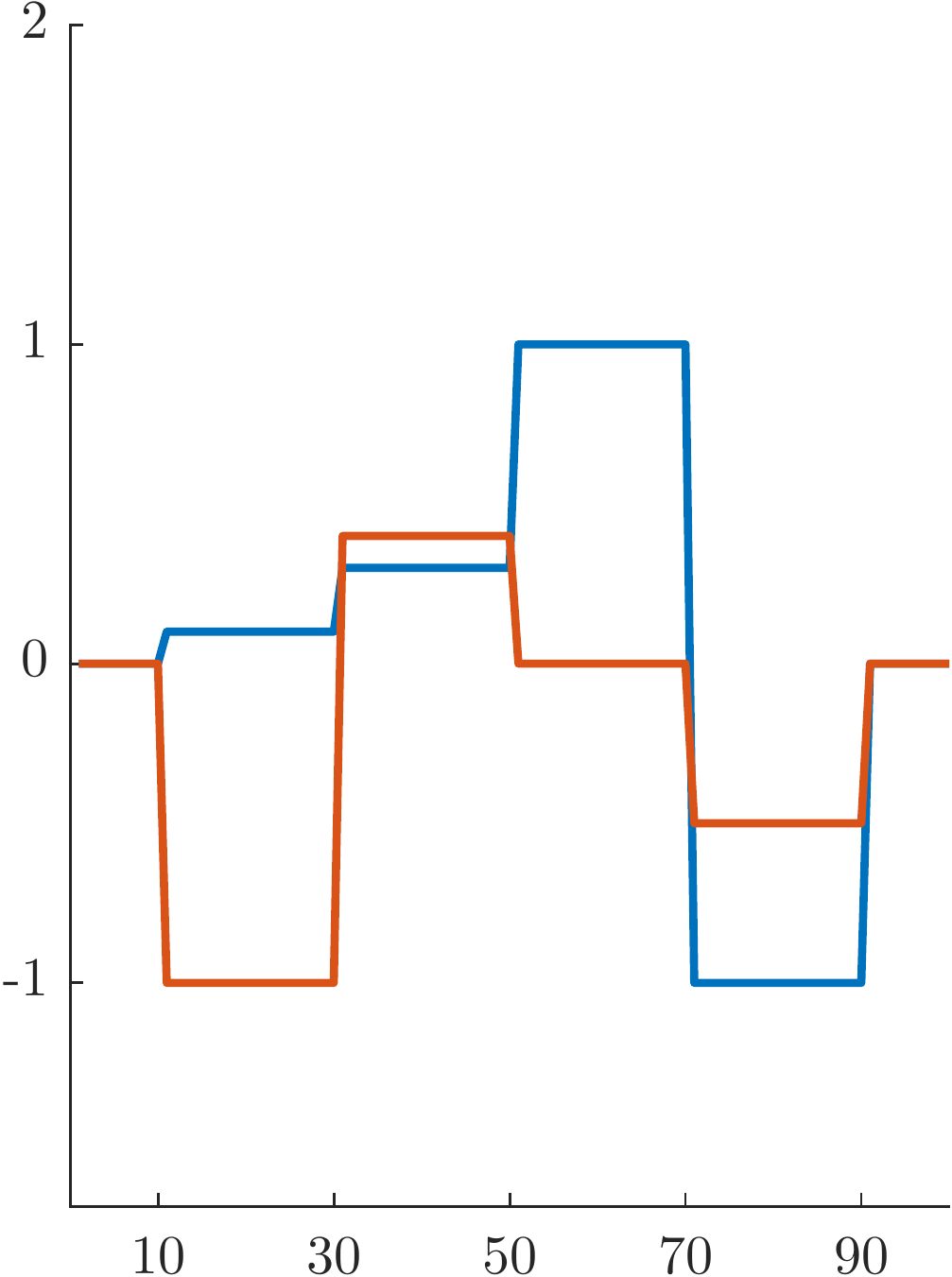} & 
   \includegraphics[height=5cm, width=0.3\textwidth]{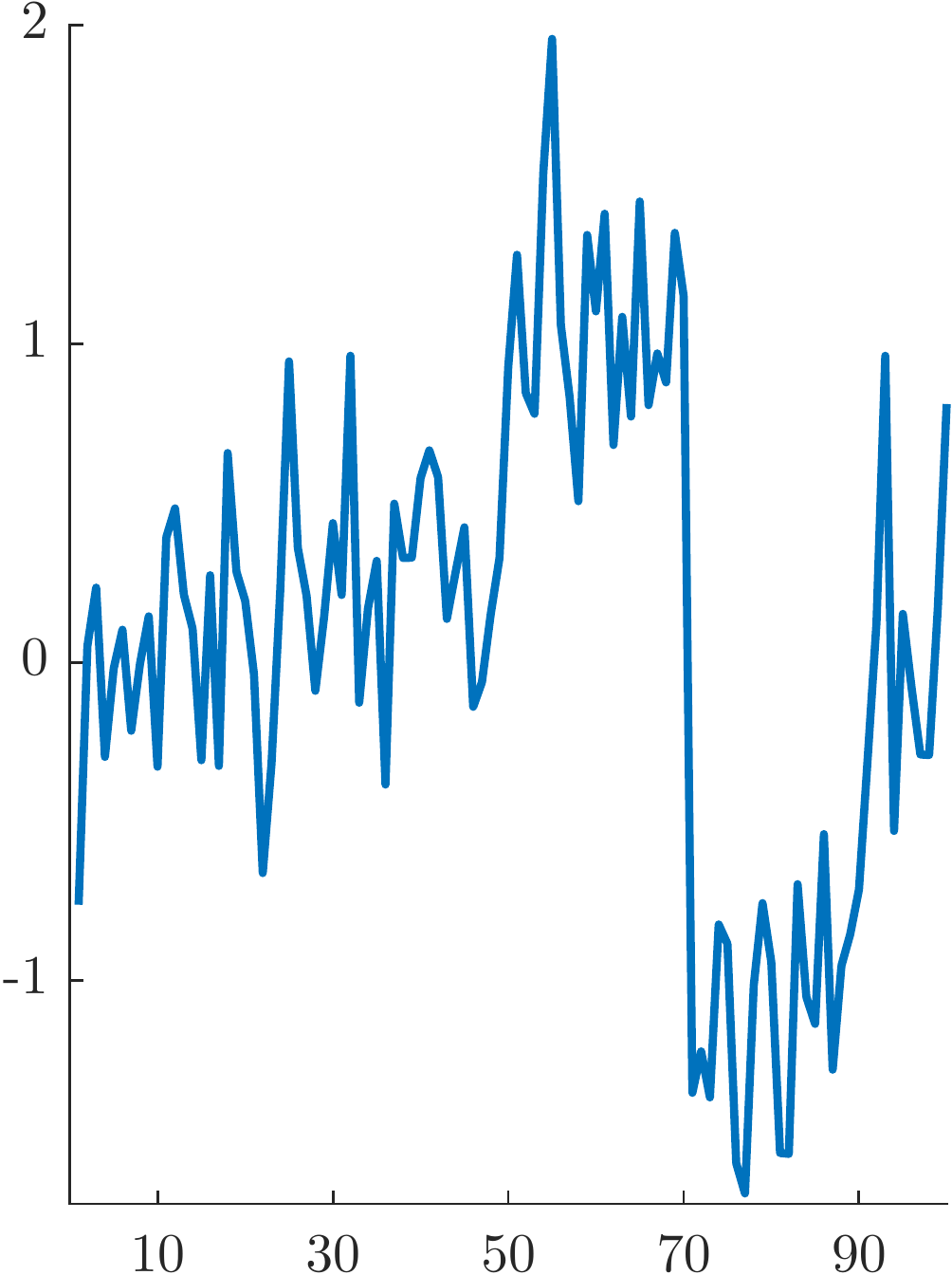} &
   \includegraphics[height=5cm, width=0.3\textwidth]{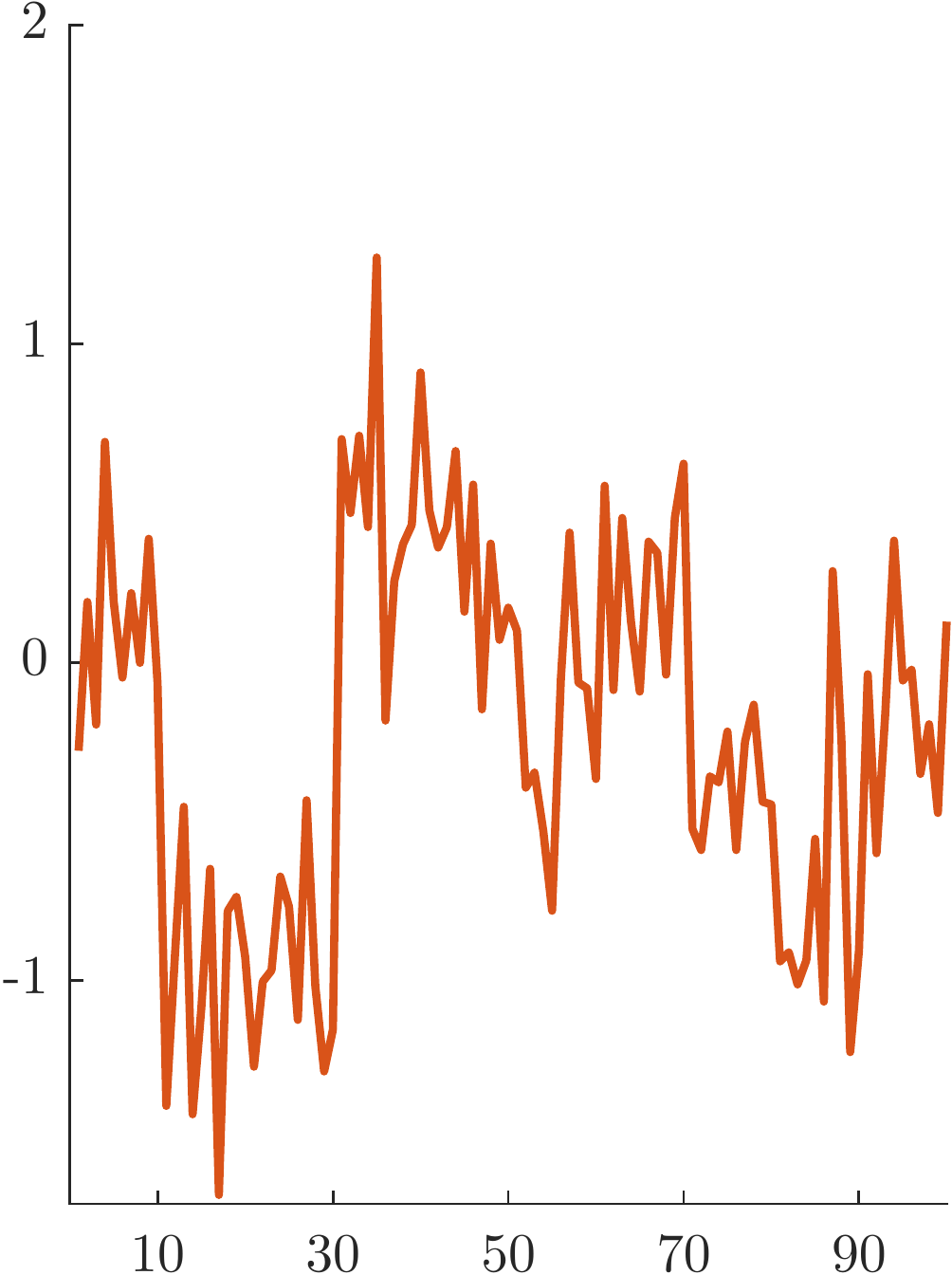} \\
   (a) Clean signals & (b) Noisy blue signal & (c) Noisy red signal
 \end{tabular}
 \caption{Example in one dimension: Clean signals and noisy signals corrupted by additive Gaussian noise with zero mean and standard deviation $\sigma = 0.35$. The edges in the clean signal are located at positions $10,30,50,70$ and $90$.}
 \label{fig:1Dexample}
\end{figure}
\begin{proof}
 By the absolute one-homogeneity of $\TV$ we have that 
 \begin{align*}
  D_\TV^{p_j^k}(c_i u_i,u_j^k) = c_i D_\TV(u_i,u_j^k) \quad\text{and} \quad \ICB^{p_j^k}(c_i u_i,u_j^k) = c_i \ICB^{p_j^k}(u_i,u_j^k),
 \end{align*}
 i.e. both functionals are positively one-homogeneous.
 We immediately conclude that 
 \begin{alignat*}{3}
  &&&\argmin_{u_i \in \BV} ~ \alpha_i H_{f_i}(K_i u_i) &&+ \sum_{j=1}^N w_{ij} D_{\TV}^{p_j^k}(u_i, u_j^k) \\ 
  &= &&\argmin_{u_i \in \BV} ~ \frac{\alpha_i}{c_i^r} H_{c_i f_i}(K_i c_i u_i) &&+ \sum_{j=1}^N \frac{w_{ij}}{c_i} D_{\TV}^{p_j^k}(c_i u_i, u_j^k)\\
  &= \frac{1}{c_i} && \argmin_{u_i \in \BV} ~ \frac{\alpha_i}{c_i^{r-1}} H_{\tilde{f}_i}(K_i u_i) &&+ \sum_{j=1}^N w_{ij} D_{\TV}^{p_j^k}(u_i, u_j^k).
 \end{alignat*}
 The proof for \eqref{eq:JointReconstructionScheme2} follows analogously.
\end{proof}
We mention that the assumptions on the data terms $H_{f_i}$ are easy to verify for quadratic data terms and the Kullback-Leibler divergence. 

\subsection{One-dimensional illustration}
As a first proof of concept, let us have a look at a one-dimensional example with two channels, featuring all the common issues of a typical joint reconstruction setting. 
Figure \ref{fig:1Dexample}(a) shows two piecewise constant signals, whose edges are located at the exact same positions, namely $10,30,50,70$ and $90$. 
However, only some of the jumps across the edges share the same direction, which is either a ``jump up'' or a ``jump down'' in both channels, while the others have opposite direction. 
The former case occurs at positions $30$, $70$ and $90$, the latter at positions $10$ and $50$. 
We would like to put special emphasis on the varying height of jumps in both channels. 
Since both signals feature small and large jumps at different positions, an overall scaling of the signals in order to have an approximately equal height of jumps everywhere is {\it not} possible. 
For example, scaling the first jump at position $10$ to equal height requires to ``shrink'' the red signal, which however leads to unreasonably large differences e.g. at location $70$. 
 \begin{figure}[t!]
 \center
 \begin{tabular}{ccc}
   \includegraphics[height=5cm, width=0.3\textwidth]{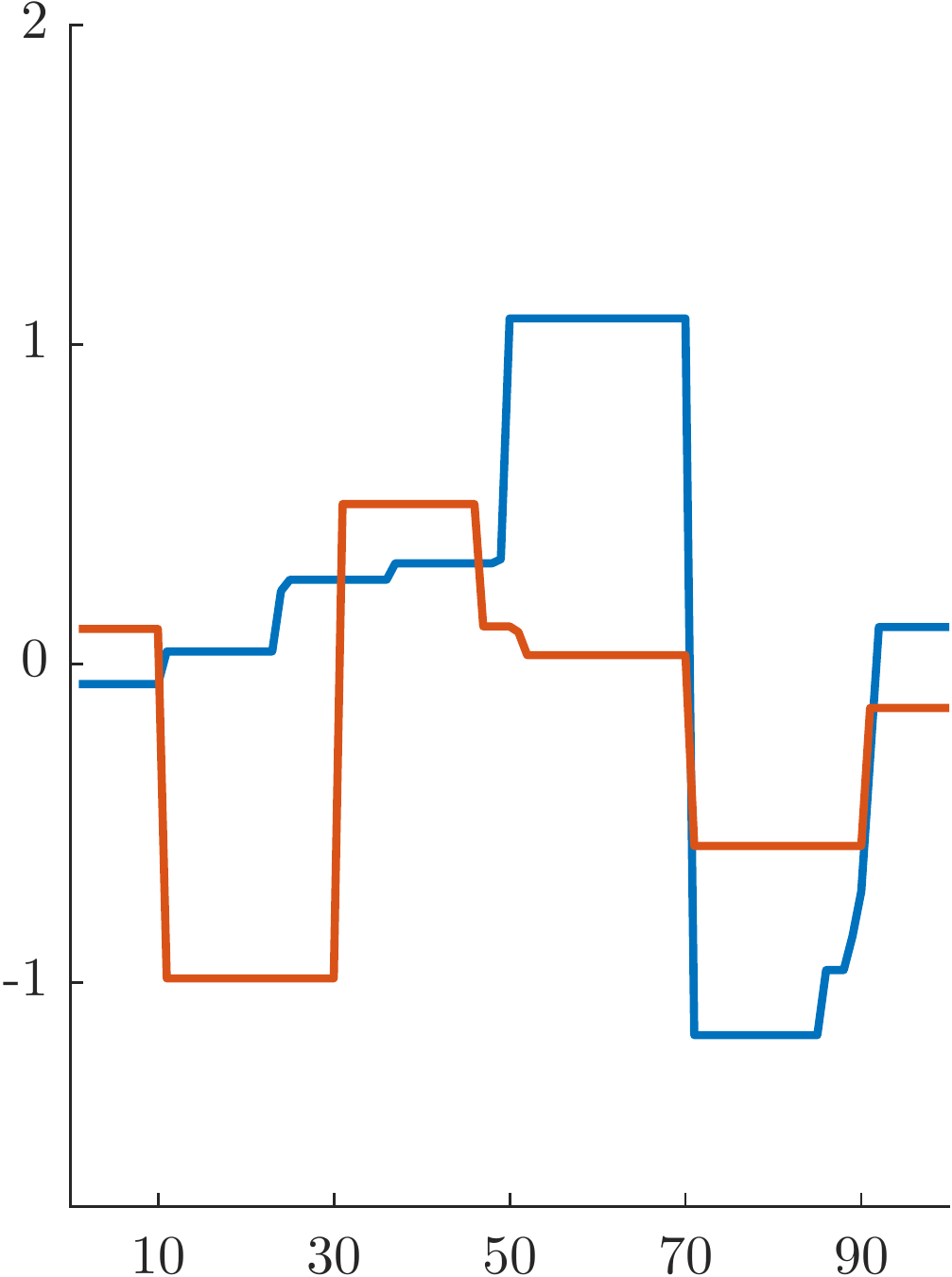} & 
   \includegraphics[height=5cm, width=0.3\textwidth]{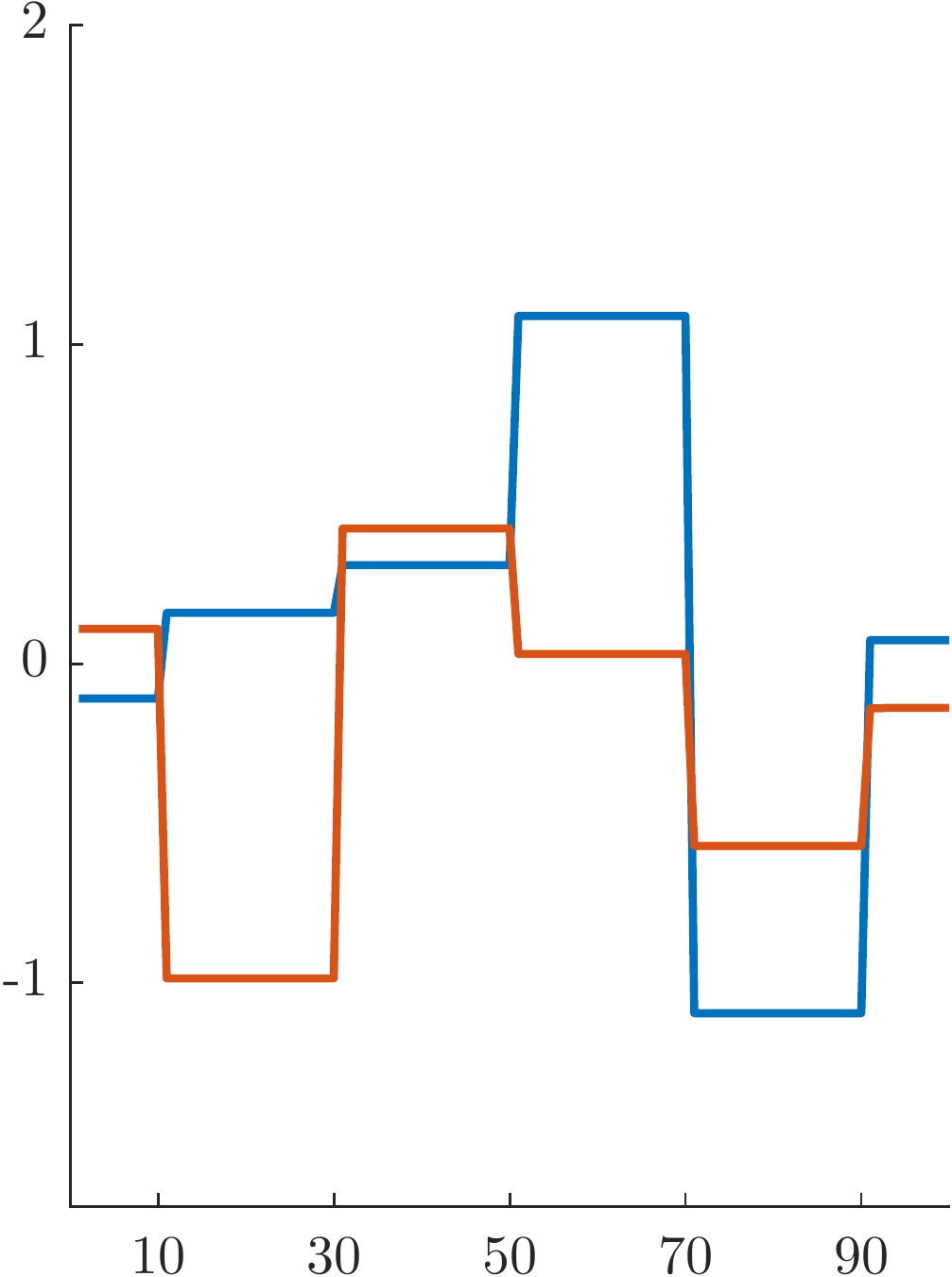} &
   \includegraphics[height=5cm, width=0.3\textwidth]{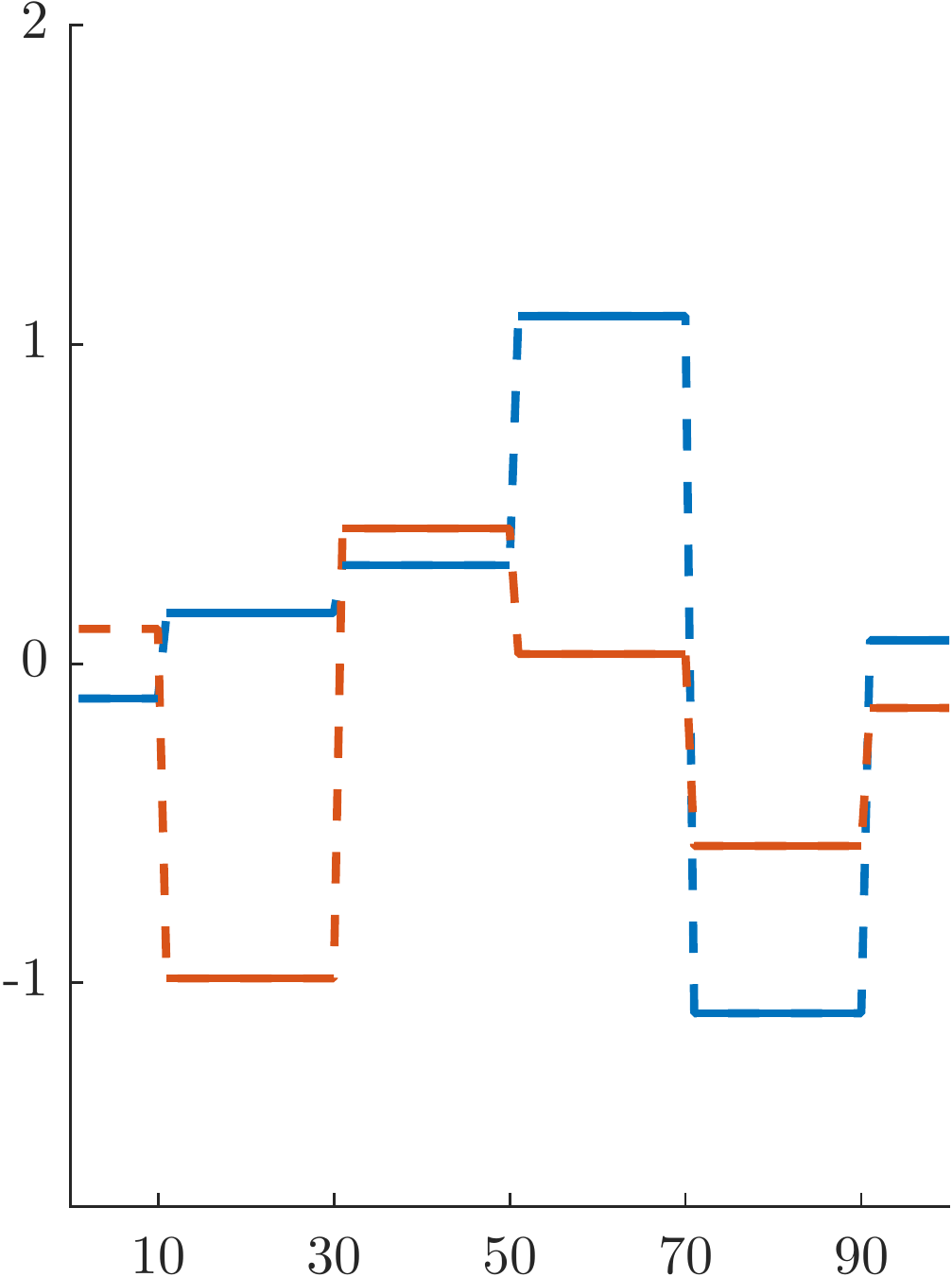} \\
   (a) Bregman TV & (b) Joint reconstruction & (c) Piecewise mean value 
 \end{tabular}
 \caption{Example in one dimension: Comparison of separate (a) and joint reconstruction (b). 
 Figure (c) shows the mean value of the noisy signals between the (known) edge locations as a comparison to (b).
 The parameters for the joint reconstruction are $\lambda = 0.5$, $\mu = 0.33$ and $7$ Bregman iterations, 12 Bregman iterations for Bregman TV.}
 \label{fig:1Dexample_results}
\end{figure}

Both signals have been artificially corrupted by additive Gaussian noise with zero mean and standard deviation $\sigma = 0.35$, which can be seen in Figure \ref{fig:1Dexample}(b) and (c).
Undoubtedly it is basically impossible to recover both signals without any further knowledge, since especially in the blue channel the noise covers the two small jumps of the signal. 
But since we know that the edges are located at the same positions, we can apply our method to couple the edge positions during the reconstruction.
Due to the Gaussian noise, the associated joint reconstruction method is: 
\begin{alignat*}{5}
u_{\mathrm{blue}}^{k+1} \in \arg \min_u & ~ \Big\{ \dfrac{\alpha}{2} \| u - f_{\mathrm{blue}} \|_2^2 && + \lambda D_{\TV}^{p_{\mathrm{blue}}^k}(u, u_{\mathrm{blue}}^k) && + (1-\lambda) \ICB^{p_{\mathrm{red}}^k}(u,u_{\mathrm{red}}^k) &&\Big\}, \\
u_{\mathrm{red}}^{k+1} \in \arg \min_u & ~ \Big\{ \dfrac{\beta}{2} \| u - f_{\mathrm{red}} \|_2^2 && + \mu D_{\TV}^{p_{\mathrm{red}}^k}(u, u_{\mathrm{red}}^k) && + (1-\mu) \ICB^{p_{\mathrm{blue}}^k}(u,u_{\mathrm{blue}}^k) && \Big\},
\end{alignat*}
with weights $\lambda,\mu \in [0,1]$. 
Note that in one dimension we indeed only couple the edge positions, since the ``direction'' of jumps across the edges is limited to ``up'' or ``down'', which is however excluded by the infimal convolution.  
Without any doubt an ordinary TV reconstruction would not be able to recover the original signal from the noisy data. 
Hence, as a comparison to the joint reconstruction result we provide individual reconstructions of both channels with a Bregman TV prior, which corresponds to $\lambda,\mu = 1$ and decouples the channels.
Since we iteratively include edge information at least from the own channel, chances are higher to find the right location of edges and the right height of jumps. 
For both separate and joint reconstructions, the crucial question is how many Bregman iterations to perform. 
For the former we chose to iterate until all edges that are present in the original signal became visible. 
The result can be found in Figure \ref{fig:1Dexample_results}(a). 

Though the denoising was successful and the individual appearance of the signals is acceptable, a large part of the edges is located at wrong positions and even some new edges are introduced to the signal.  
Even worse, comparing between channels, the assumption of equal edge positions is not met. 
In contrast, the joint reconstruction result (b) after $7$ Bregman iterations shows a perfect recovery of the edge sets of both signals. 
With regard to the intensities of the signal, or the height of the jumps, we obtain slightly different values than for the original signal. 
This is however not surprising, since the mean of this specific noise realization is not exactly zero. 
And indeed, taking the mean value of the noisy signals between the (known) edge locations in Figure \ref{fig:1Dexample_results}(c) reveals the exact same signal as the joint reconstruction, which confirms that the joint reconstruction is a perfect recovery of the data over the true edge set.  

A closer look at the infimal convolution delivers further insight into its behavior.
Recall that we have e.g. for the blue channel
\begin{align*}
 \ICB^{p_{\mathrm{red}}^k}(u_{\mathrm{blue}},u_{\mathrm{red}}^k) = \inf_{z_{\mathrm{blue}}} ~ D_{\TV}^{p_{\mathrm{red}}^k}(u_{\mathrm{blue}}-z_{\mathrm{blue}}, u_{\mathrm{red}}^k) + D_{\TV}^{-p_{\mathrm{red}}^k}(z_{\mathrm{blue}}, -u_{\mathrm{red}}^k).
\end{align*}
Here we seek for a minimizing decomposition of the blue signal into two parts, such that one part matches the subgradient $p_{\mathrm{red}}^k$ and the other part matches the negative subgradient $-p_{\mathrm{red}}^k$. 
To do so, $z_{\mathrm{blue}}$ has to compensate for jumps at positions where the two channels do not share the direction of the jump, i.e. at locations $10$ and $50$.
Figure \ref{fig:1Dexample_infConv}(a) shows the first part of the decomposition for both channels, namely $\mathbf{u}-\mathbf{z}$. 
Intuitively, the signal $\mathbf{u}-\mathbf{z}$ should only have jumps at locations where the signals share edges with the same direction and otherwise be constant, and indeed this is the case. 
In contrast, the variable $\mathbf{z}$ itself has to match the negative subgradient of the {\it other} channel, i.e. it has to have the same jump locations and directions as the negative signal from the other channel.
This can be seen in Figure \ref{fig:1Dexample_infConv}(b) and (c). 
Note that the signals do not have to share every edge, but also constant signals are possible which can be seen e.g. in (b) at location $30$.

We continue by showing how to solve the joint reconstruction problem numerically, before we get to a more sophisticated numerical example, namely PET-MRI joint reconstruction.
 \begin{figure}[t!]
 \center
 \begin{tabular}{ccc}
   \includegraphics[height=5cm, width=0.3\textwidth]{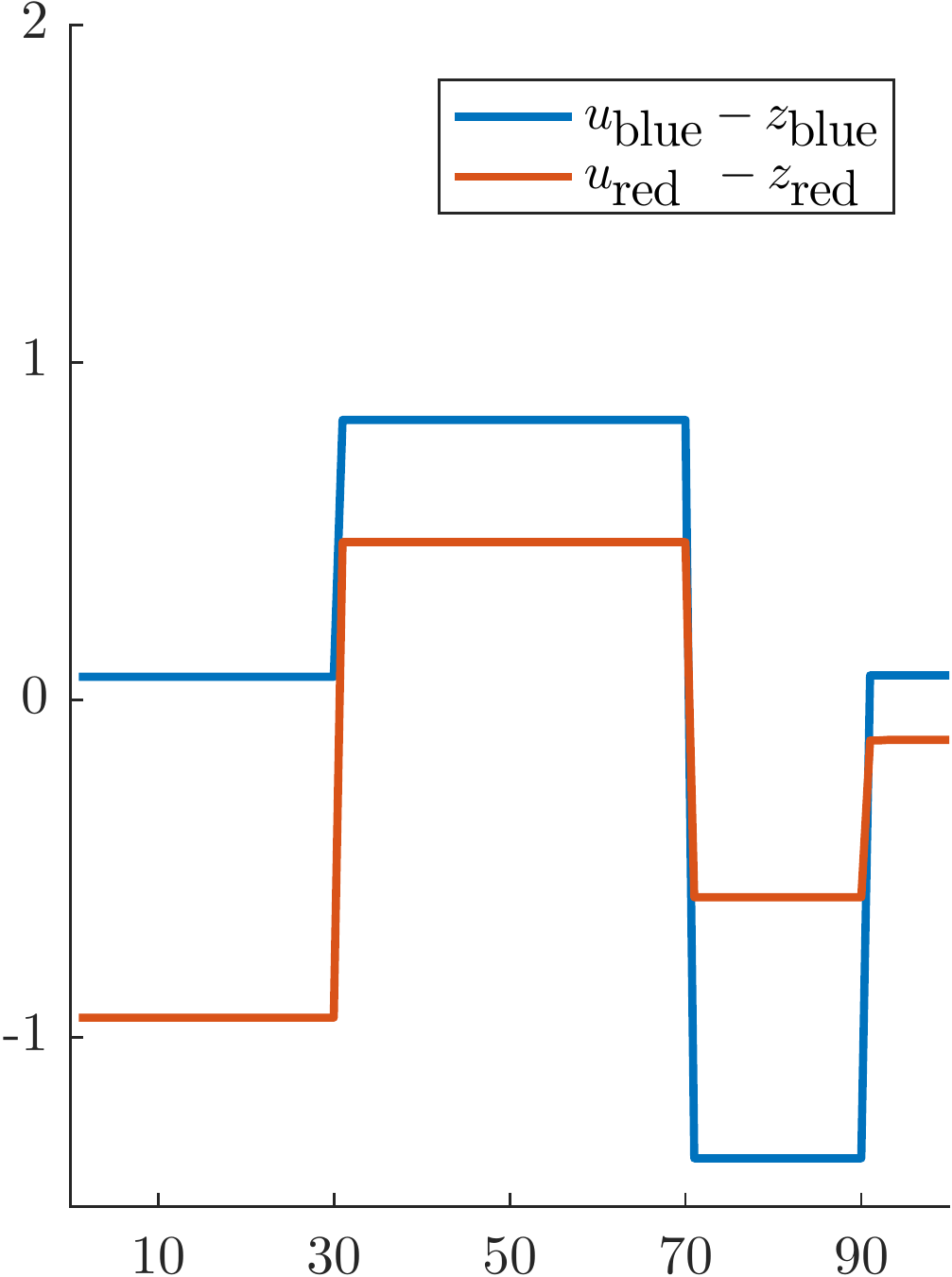} & 
   \includegraphics[height=5cm, width=0.3\textwidth]{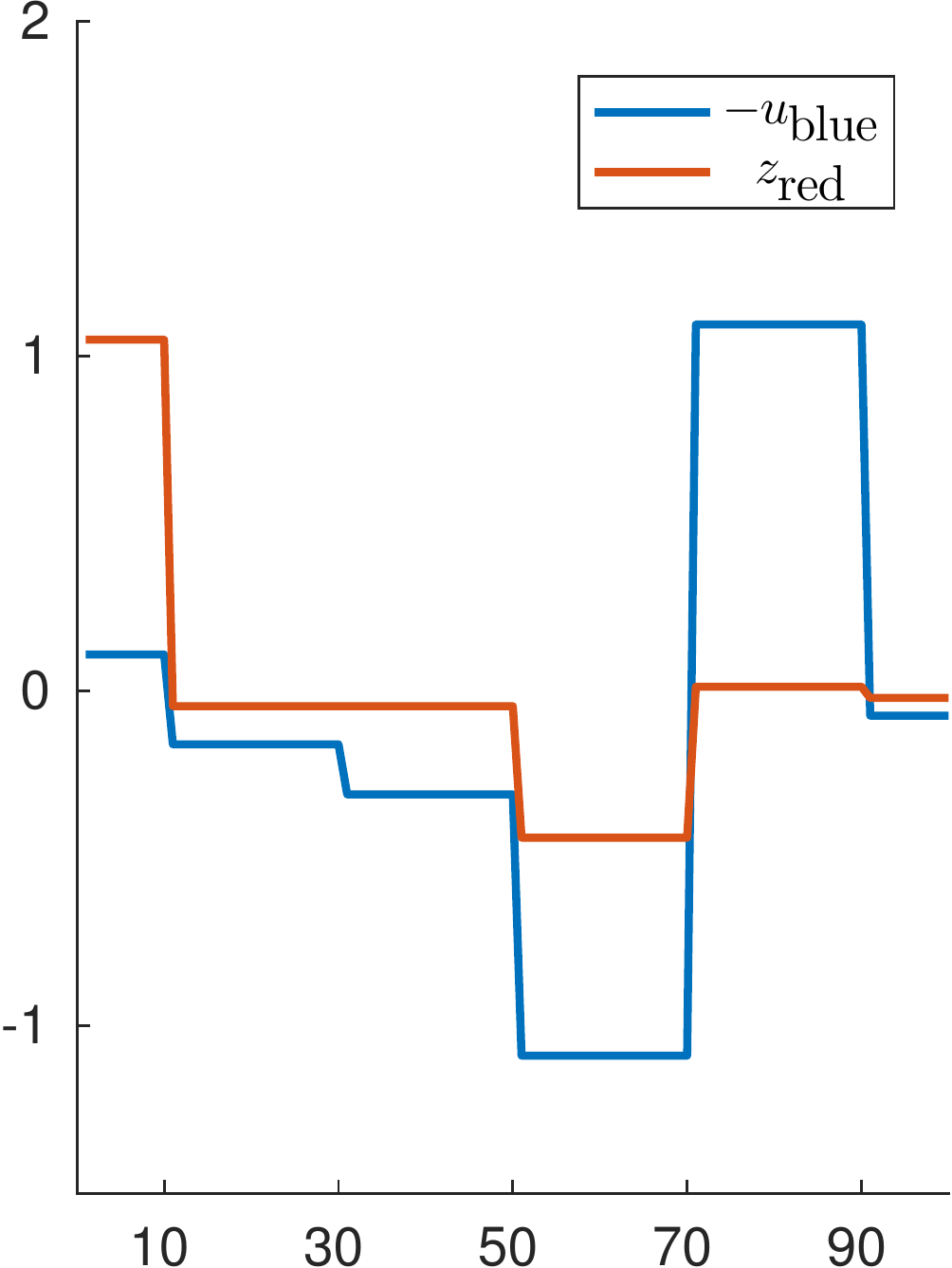} &
   \includegraphics[height=5cm, width=0.3\textwidth]{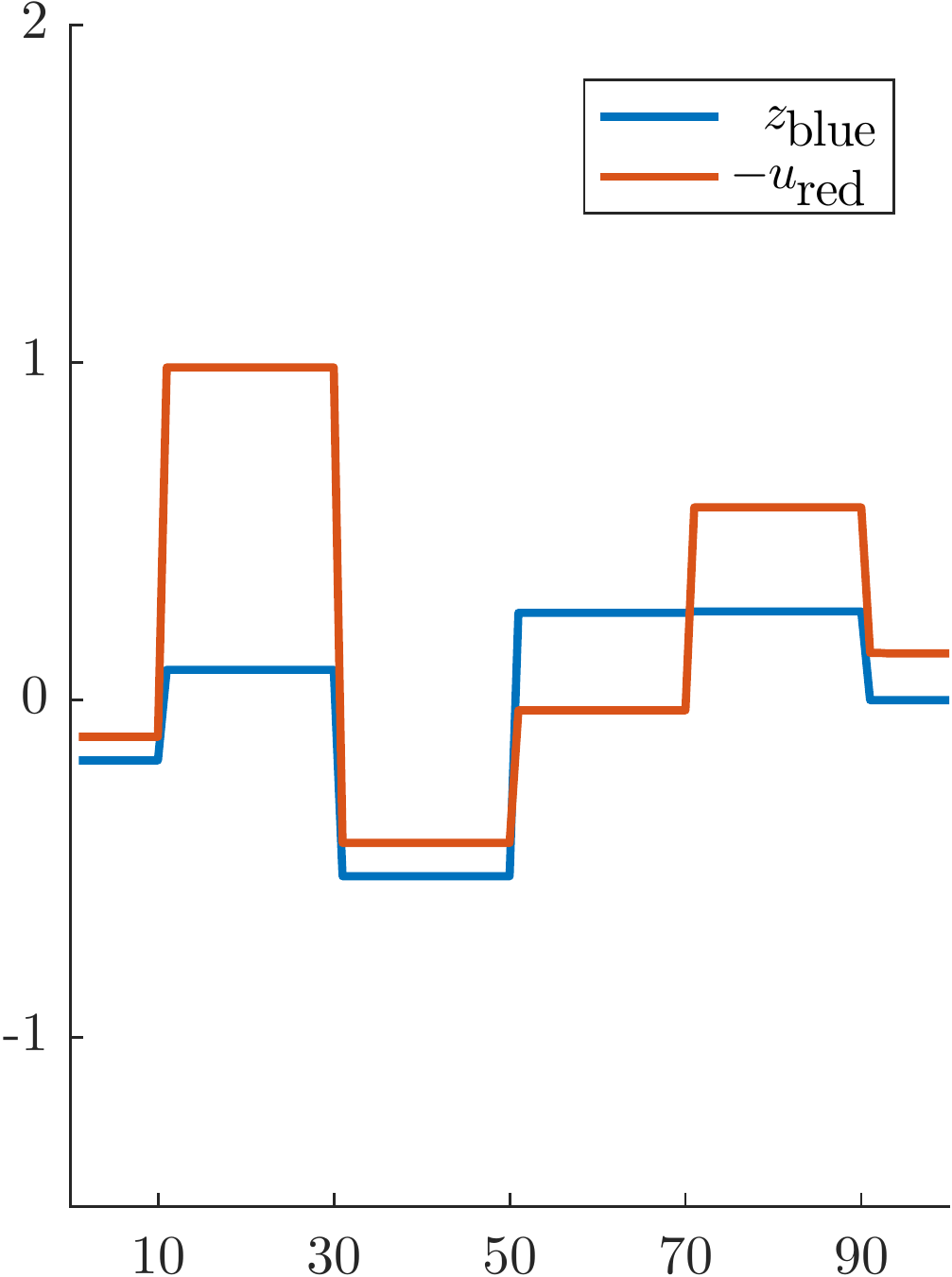} \\
   (a) $\mathbf{u}-\mathbf{z}$ & (b) $-u_{\mathrm{blue}}$ and $z_{\mathrm{red}}$ & (c) $-u_{\mathrm{red}}$ and $z_{\mathrm{blue}}$
 \end{tabular}
 \caption{Example in one dimension: Infimal Convolution.}
 \label{fig:1Dexample_infConv}
\end{figure}

%% file: numericalsolution.tex
\section{Numerical solution}\label{sec:numerics}
In this section we show how to solve the joint reconstruction problem numerically. 
Due to the non-differentiability of the involved Bregman distances we employ a primal-dual method (cf. e.g. \cite{Pock:MinimizingMumfordShah, Esser:GeneralFramework, ChambollePock}).
In general, the idea is to dualize every term of \eqref{eq:JointReconstructionScheme2} containing an operator using the notion of convex conjugates \cite{Ekeland:ConvexAnalysis}
\begin{align*}
h^*(y) = \sup_x ~ \langle y,x \rangle - h(x). 
\end{align*}
The goal is to obtain a saddle-point problem such that the proximal maps
\begin{align*}
\prox_{\gamma h}(y) = \arg \min_x ~ \dfrac{1}{2 \gamma} \|x - y\|_2^2 + h(x)
\end{align*}
for all the involved functions are easy to solve via point-wise operations or simple projections. 
This is e.g. the case for the characteristic function $\delta_S$ of a convex set $S$, i.e. 
\begin{align*}
\delta_S(y) = \begin{cases}
0 & y \in S \\
+ \infty & y \notin S,
\end{cases}
\end{align*}
where 
\begin{align*}
	\prox_{\gamma \delta_S}(y) = \proj_S(y).
\end{align*}
The advantage of the saddle-point formulation is that the proximal maps of dualized terms decouple, which allows us to treat every Bregman distance separately.
Since every additional channel simply introduces another infimal convolution we restrict ourselves to two channels and comment on the straight-forward extension to an arbitrary number of channels at the end of the section.  
We nevertheless use a notation with indices $i$ and $j$ rather than $1$ and $2$, in order to unify the derivation for later use and to make the notation as simple as possible. 
The channel which is currently reconstructed is referred to by $i$, while the foreign channel related to the infimal convolution is indexed by $j$.
So let us consider two inverse problems
\begin{align*}
K_i u_i = f_i
\end{align*}
$u_i \in \R^N$, $f_i \in \C^{M_i}$, $K_i \in \R^{M_i \times N}$ for $i = 1,2$, that we want to solve jointly for $u_i \in C_i$, where $C_i \subset \R^N$ are convex sets. 
We discretize the two dimensional gradient $\nabla \colon \R^N \to \R^{N \times 2}$ using standard forward differences (see e.g. \cite{ChambollePock}) and define the (isotropic) total variation of $u \in \R^N$ as 
\begin{align*}
 \TV(u) := \| \nabla u \|_1 = \sum_{i = 1}^{N} \sqrt{|(\nabla u)_{i,1}|^2 + |(\nabla u)_{i,2}|^2}.
\end{align*}
We make use of the specific structure of the subdifferential of $\TV$.
By the chain rule we find that $p \in \partial \TV(u)$ if and only if there exists $q \in \partial \|\cdot\|_1(\nabla u)$ such that $p = - \nabla \cdot q$.
Hence, let us first assume that we are already given a subgradient 
\begin{align*}
p_i^k = -\nabla \cdot q_i^k \in \partial \TV(u_i^k) 
\end{align*}
of the current approximations $u_i^k$. 
The next step of the method is then to solve
\begin{align*}
u_i^{k+1} &\in \arg \min_u ~ \alpha_i H_{f_i}(K_i u) + w_i D_{\TV}^{p_i^k}(u,u_i^k) + (1-w_i) \ICB^{p_j^k}(u,u_j^k) + \delta_{C_i}(u).
\end{align*} 
In order to derive the saddle-point formulation of the problem we have to dualize every involved term except for the characteristic function. 
Recall that we cannot access a general closed form representation of $\ICB$. 
\begin{algorithm}[t!]
\caption{\textbf{Two-channel joint reconstruction (One step)}}
{
\begin{algorithmic}[1]
\Require $f_i$, $w_i$, $q_i^k,q_j^k$, $\tau,\sigma > 0$
\Ensure $u = \bar{u} = K_i^*f_i, z = \bar{z} = 0, \; y_1 = y_2 = y_3 = y_4 = 0$
	\While{$\sim$ stop crit}
		\State {\it Dual updates}
		\State $y_1 \leftarrow \prox_{H_{f_i}^*}(y_1 + \sigma K_i \bar{u} )$
		\State $y_2 \leftarrow \proj_{S_2}(y_2 + \sigma \nabla \bar{u})$
		\State $y_3 \leftarrow \proj_{S_3}(y_3 + \sigma \nabla (\bar{u} - \bar{z}))$
		\State $y_4 \leftarrow \proj_{S_4}(y_4 + \sigma \nabla \bar{z})$
		\State {\it Primal updates}
		\State $\tilde{u} \leftarrow \proj_{C_i} ( u - \tau (K_i^* y_1 - \nabla \cdot y_2 - \nabla \cdot y_3) )$
		\State $\tilde{z} \leftarrow z - \tau (\nabla \cdot y_3 - \nabla \cdot y_4)$
		\State {\it Overrelaxation}
		\State $(\bar{u}, \bar{z}) \leftarrow 2 (\tilde{u}, \tilde{z}) - (u,z)$
		\State $(u,z) \leftarrow (\tilde{u}, \tilde{z}) $		
	\EndWhile\\
\Return $u = u_i^k$
\end{algorithmic}
}
\label{alg:TwoChannel}
\end{algorithm}
Instead, we use the definition of the infimal convolution to introduce an additional auxiliary primal variable $z$ into the problem:
\begin{align}\label{eq:1step}
	&\min_{u,z} ~ H_{f_i}(K_i u) + w_i D_{\TV}^{p_i^k}(u,u_i^k) + (1-w_i) D_{\TV}^{p_j^k}(u-z,u_j^k) + (1-w_i) D_{\TV}^{-p_j^k}(z,-u_j^k) + \delta_{C_i}(u).
\end{align}
We observe that the subproblem we have to solve in every iteration basically consists of only two parts: A (differentiable) data term and a sum of total variation Bregman distances with regard to different subgradients.
Due to the structure of their subgradients, the Bregman distances can be dualized as a shifted and scaled $\ell_1$-norm, which we show at the example of the first arising Bregman distance
\begin{align*}
	&\min_u ~ w_i D_{\TV}^{p_i^k}(u,u_i^k)
	= \min_u ~ w_i ( \|\nabla u\|_1 - \langle p_i^k, u \rangle ) 
	= \min_u ~ w_i \|\nabla u\|_1 - w_i \langle q_i^k, \nabla u \rangle \\
	= &\min_{u,b} \max_y ~ \langle y, \nabla u - b \rangle + w_i \|b \|_1 - \langle w_i q_i^k, b \rangle 
	= \min_u \max_y ~ \langle y, \nabla u \rangle - \max_b ~ \Big(\langle y + w_i q_i^k,b \rangle - w_i \|b\|_1 \Big) \\
	= &\min_u \max_y ~ \langle y, \nabla u \rangle - \delta_S(y). 
\end{align*}
Here, the set $S$ is defined as
\begin{align*}
	S = \{ s ~|~ \| s + w_i q_i^k \|_{\infty} \leq w_i \}.
\end{align*}
The proximal map for the characteristic function $\delta_S$ can be computed explicitly as a projection onto the set $S$, or a shifted projection onto the set $\tilde{S} = \{s ~|~ \| s \|_{\infty} \leq w_i \}$:
\begin{align*}
\prox_{\gamma \delta_S} (y) = \proj_S(y) = \proj_{\tilde{S}}(y + w_i q_i^k) - w_i q_i^k.
\end{align*}
Proceeding analogously with the two other Bregman distances and dualizing the data term we end up with the following primal-dual formulation: 
\begin{align}
\min_{u,z} \max_{y_1,\dots,y_4} &~ \langle y_1,K_i u \rangle + \langle y_2, \nabla u \rangle + \langle y_3,\nabla (u-z) \rangle + \langle y_4, \nabla z \rangle \nonumber \\
&- H_{f_i}^*(y_1) - \delta_{S_2}(y_2) - \delta_{S_3}(y_3) - \delta_{S_4}(y_4) + \delta_{C_i}(u),
\label{eq:PrimalDualNumericalProblem}
\end{align}
where 
\begin{align*}
S_2 &= \{s ~|~ \|s + w_i q_i^k \|_{\infty} \leq w_i \}, \\
S_3 &= \{s ~|~ \|s + (1-w_i) q_j^k \|_{\infty} \leq (1-w_i) \}, \\
S_4 &= \{s ~|~ \|s - (1-w_i) q_j^k \|_{\infty} \leq (1-w_i) \}.
\end{align*} 
Hence, all the proximal updates for the involved Bregman distances can be carried out as a projection on one of the sets $S_l$. 
We illustrate one step of the numerical algorithm in Algorithm \ref{alg:TwoChannel} and summarize the whole reconstruction process in Algorithm \ref{alg:Whole}.
In order to avoid too many indices, we drop the Bregman iteration number $k$ and still focus on only one channel. 
The other one follows analogously by simply exchanging the roles of $i$ and $j$. 

The extension to more than two channels is then easily done.
Since an additional channel adds a second infimal convolution to the problem, we can introduce another auxiliary primal variable to obtain two additional Bregman distances. 
These can be treated as before by a projection onto the sets associated to the involved subgradients. 


\subsection{Stopping criteria}
For appropriate results it is necessary to decide when to consider the single subproblems \eqref{eq:1step} as converged. 
We use a modified primal-dual gap in the following sense: 
the dual problem of \eqref{eq:1step} (respectively \eqref{eq:PrimalDualNumericalProblem}) is given by 
\begin{align*}
 \max_{y_1, \dots, y_4} ~ - H_{f_i}^*(y_1) \quad \text{s.t.} \quad \begin{cases}
                                                              y_j \in S_j,\quad  j = 2,3,4, \\
                                                              K_i^* y_1 - \nabla \cdot y_2 - \nabla \cdot y_3 \in C_i, \\
                                                              \nabla \cdot y_3 - \nabla \cdot y_4 = 0.
                                                             \end{cases}
\end{align*}
We hence define a primal-dual gap for the $k$-th Bregman iteration of the $i$-th channel as 
\begin{align*}
 \G_i^k(u,z,y_1) = H_{f_i}(K_i u) + w_i D_{\TV}^{p_i^k}(u,u_i^k) + (1-w_i) D_{\TV}^{p_j^k}(u-z,u_j^k) + (1-w_i) D_{\TV}^{-p_j^k}(z,-u_j^k) + H_{f_i}^*(y_1),
\end{align*}
which has to converge to zero.
Note that due to the projection onto $S_j$ in every step of Algorithm \ref{alg:TwoChannel}, the constraint on $y_j$ for $j=2,3,4$ is trivially fulfilled. 
The latter two constraints however have to be checked separately alongside the gap. 
We relaxed the third constraint and normalized the gap by the number of primal pixels $N$ such that we consider the subproblem as converged if 
\begin{align*}
 &\G_i^k(u,z,y_1) / N \leq tol_{gap}, \quad
 K_i^* y_1 - \nabla \cdot y_2 - \nabla \cdot y_3 \in C_i, \quad
 \| \nabla \cdot y_3 - \nabla \cdot y_4 \| \leq tol_{constraint}.
\end{align*}
We remark that the evaluation of the gap does not require any severe computational costs since the application of the imaging operator $K_i$ and the gradients have to be done for the next iteration anyway.
It is also worth mentioning that in practice the second constraint seemed to be the best measure of convergence, since it was usually fulfilled last. 
Interestingly it can even serve as an indicator, which pixels in $u$ have not converged yet, since they directly correspond to the positions not fulfilling the second constraint. 

\subsection{Subgradient updates}
After convergence of both channels, we need to update the subgradient for the next Bregman iteration. 
Unfortunately we cannot find an update equation equivalent to (\ref{eq:ColorBregmanSubgradients}) directly from the optimality condition of the problem due to the involved infimal convolution.
But since we can access the dual variables, we can derive an alternative subgradient update. 
Note that for a stationary point $((u,z),(y_1,y_2,y_3,y_4))$ of problem \eqref{eq:PrimalDualNumericalProblem} the second dual variable $y_2$ needs to fulfill
\begin{align*}
	\nabla u \in \partial \delta_{S_2}(y_2).
\end{align*}
Recall that $\delta_{S_2}$ is the convex conjugate of a Bregman distance which is  proper, convex and lower semicontinuous. 
By duality, the above inclusion is hence equivalent to 
\begin{align*}
	y_2 \in \partial [w_i D_\TV^{p_i^k}(u,u_i^k)] \Leftrightarrow y_2 \in \partial \Big[ w_i \|\nabla u\|_1 - w_i \langle q_i^k, \nabla u \rangle \Big] \Leftrightarrow \dfrac{y_2}{w_i} \in \partial \|\nabla u\|_1 - q_i^k,
\end{align*}
which yields the subgradient update 
\begin{align}
	q_i^{k+1} = \dfrac{y_2}{w_i} + q_i^k \in \partial \|\nabla u\|_1 = \partial \TV(u). 
	\label{eq:SubgradientUpdate}
\end{align}
It is important to notice that the convergence of the algorithm obviously influences the quality of the obtained subgradient. 
In particular it can happen that a poor approximation does not even fulfill the requirements of a subgradient \eqref{eq:l1_subdifferential}, i.e. $|q_i^{k+1}|_l > 1 $ for some index $l$. 
This can be avoided by a sufficient convergence of the subproblem, or can even be manually corrected for if necessary.
With the tolerances chosen in our numerical studies however we did not experience any problems. 

\begin{algorithm}[t!]
	\caption{\textbf{Two-channel joint reconstruction}}
	{
	\begin{algorithmic}[1]
		\Require data $(f,g)$, weights $(\lambda,\mu)$, regularization parameters $(\alpha, \beta)$, $\#$ Bregman iterations $K$
		\Ensure $u^0, v^0 = 0$, $q_u^0, q_v^0 = 0$
		\For{$k=0$ to $K$}  
		\State Compute $u^{k+1}$ via Algorithm \ref{alg:TwoChannel}
		\State Update the subgradient $q_u^{k+1}$ via \eqref{eq:SubgradientUpdate}
		\State Compute $v^{k+1}$ via Algorithm \ref{alg:TwoChannel}
		\State Update the subgradient $q_v^{k+1}$ via \eqref{eq:SubgradientUpdate}
		\EndFor \\
		\Return $u^K, v^K$		
	\end{algorithmic}
	}
\label{alg:Whole}
\end{algorithm}

%% file: petmri.tex
\section{PET-MR imaging and numerical results}\label{sec:petmri}
In the following we give some further motivation for the use of structural joint reconstruction using a topic of very recent interest, namely PET-MR imaging, which is the example for the numerical studies at the end of this section.
Similar studies can be found e.g. in \cite{Ehrhardt:JointReconstruction,Knoll2017}.

\subsection{PET-MRI modeling}
While widely established imaging techniques such as PET-CT only allow for a sequential data acquisition, the coupling of PET and MR imaging devices offers further opportunities of mathematical reconstruction and finally clinical diagnosis. 
Since these scanners are able to \textit{simultaneously} acquire functional PET data and structural MR data they provide access to spatially and temporally registered data of two complementing imaging techniques. 
PET imaging is capable of showing metabolic processes and is therefore e.g. useful to locate tumors and diseased areas of the heart, but unfortunately suffers from a poor resolution and a lack of anatomical information. 
In contrast, MR imaging yields a very high resolution with great contrast between different types of tissue, but provides only little possibilities to display metabolic processes and realize quantitative imaging. 
Additionally, it is common practice to speed up the data acquisition for MRI by undersampling the $k$-space, i.e. by measuring only a fraction of the data, which degrades the quality of the resulting reconstruction. 
However, having data from both modalities at hand, one can try to exploit the additional information in order to re-improve the results, in particular for PET, but as well for MRI.  
Since we cannot expect function to be independent of the underlying anatomy \cite{Ehrhardt:JointReconstruction}, it is reasonable to assume a similar structure in both the PET and the MR image, which motivates a structural joint reconstruction.

We briefly introduce the problem of PET and MR imaging and establish our notation, restricting ourselves to a setting in two dimensions.
For a more sophisticated introduction we refer the reader e.g. to \cite{Dawood:CorrectionTechniques, WernickAarsvold} for PET and \cite{Liang:PrinciplesOfMRI, Siemens:MagnetsSpinsResonances} for MRI. \\

PET makes use of the radioactive decay of an artificially created positron emitter, so-called tracer, applied to the patient's body. 
The emitted positrons annihilate with free electrons from the surrounding tissue, thereby generating two gamma photons which travel in (almost) opposite direction. 
The pair of gamma photons can be measured by a detector ring surrounding the patient. 
If two gamma photons impact on the detector surface at approximately the same time, they are likely to originate from the same decay. 
From the positions of the impacts we thus gain information about the approximate location of the decay, which must have taken place along the connecting line of response. 
The collectivity of all lines of response mathematically corresponds to a fan-beam-type sampling of the Radon transform (cf. e.g. \cite{Natterer:MathematicalMethods}) in two dimensions, i.e. line integrals over the spatial tracer distribution between two detector pairs. 

Hence let $\Omega \subset \R^2$ be an open and bounded image domain and $u \colon \Omega \to [0, \infty )$ denote the tracer distribution. 
The imaging process is usually modeled as a semi-discrete operator equation, mapping the continuous representation of $u \in L^1(\Omega)$ to the discrete data $f \in \R^M$ \cite{Sawatzky:TotalVariation, Sawatzky:EMTVMethods}.
We restrict ourselves to the fully discrete case, so the imaging operator is a matrix $A \colon \R^N \to \R^M$ mapping $u \in \R^N$ on a finite grid $\Omega_0$ to the discrete data $f$. 
We aim to solve the inverse problem 
\begin{align}
Au=f.
\end{align}
This is essentially an integral operator equation, and as the data always contains noise we are in need of a-priori information to stabilize the reconstruction process. 
In general, the underlying random process of tracer decay and photon detection as well as noise is assumed to be Poisson distributed. 
Thus the reconstruction of $u$ is commonly carried out as the maximum a-posteriori estimator for $u$, which leads to the minimization problem 
\begin{align}\label{eq:pet_recon}
u \in \arg \min_{u \geq 0} \left\{ \alpha \sum_{i=1}^M (Au)_i - f_i \log(Au)_i + \Reg(u) \right\}, 
\end{align}
where we used a Gibbs prior $p(u) = \exp (-\alpha^{-1} \Reg(u))$ for the a-priori density of $u$. 
By adding the term $-f + f \log(f)$ (we use the convention $0 \log (0):= 0$) independent of $u$ we gain the famous Kullback-Leibler divergence, which we denote by $\DKL(f,Au)$. 
Note that the value of $\DKL(f,Au)$ has to be set to infinity if $(Au)_i = 0$ for some $i \in \{1, \dots, M\}$.
Since $u$ represents a density, it is common practice to add a positivity constraint on $u$.
For problem \eqref{eq:pet_recon} to be well-posed it is usually necessary to add some additional technical assumptions on the imaging operator $A$.
We require the operator $A$ to preserve positivity, i.e. if $u \geq 0$ then necessarily $Au \geq 0$, and refer the reader to \cite[p. 79]{Sawatzky:EMTVMethods} for a deeper discussion of the topic.

The choice of the prior $\Reg$ now determines the type of a-priori knowledge we add to the problem. 
Later, we compare our joint reconstruction method to individual reconstructions using no prior (Expectation Maximization with early stopping), Total Variation regularization (TV, see e.g. \cite{ROF, Sawatzky:EMTVMethods,Bonettini2011,Harmany2012}), as well as to joint reconstruction techniques via Joint Total Variation (JTV \cite{Haber:ModelFusion}) and Parallel Level Sets (PLS \cite{Ehrhardt:JointReconstruction}). \\

MR imaging makes use of the nuclear spins of hydrogen nuclei inside the human body. 
By applying several magnetic fields to the patient, these spins can be specifically aligned such that one can measure differences in the resulting resonant frequencies. 
Mathematically, the recorded data sets correspond to measurements of the sought-for quantity in Fourier space \cite{Liang:PrinciplesOfMRI}.
For simplicity, we ignore the phase of the MR image and restrict our studies to absolute magnitude imaging \cite{Brown:BasicPrinciples}. 
Thus the MR image is a non-negative, real-valued image $v \colon \Omega \to [0,\infty)$. 
As for PET, we restrict ourselves to the discrete case for the modeling.
Hence, the imaging operator $B$ is essentially a discrete Fourier transform which maps the discrete $v \in \R^N$ to the Fourier measurements $g \in \C^L$ and we aim to solve
\begin{align}
Bv = g.
\end{align}
Assuming the noise in MRI to be Gaussian, we can seek for $v$ as the maximum a-posteriori estimator, i.e. 
\begin{align*}
v \in \arg \min_{v \geq 0} \Big\{ \frac{\beta}{2} \sum_{j=1}^L |(Bv)_j - g_j| ^2 + \Reg(v) \Big\}.
\end{align*}
We shall denote the data term by $\| \cdot \|_\C^2$. 
It is common practice to accelerate the MR data acquisition by measuring only a sparse set of Fourier coefficients. 
The missing information is then compensated for by suitable a-priori knowledge. 
The most popular approach is compressed sensing (\cite{Candes:Robust, Candes:Stable, Donoho:CompressedSensing, Lustig:Sparse}, where the main idea is to map the MR image to some function space where it has a sparse representation, and can thus be reconstructed from substantially reduced data. 
The undersampled frequencies are usually measured randomly or, for practical reasons, on a geometric pattern such as spirals or spokes, which causes artifacts in the reconstruction. 
As for PET, we compare the results of our joint reconstruction to already existing individual and joint reconstruction techniques.  
These are a zero-filled inverse Fourier transform (no prior), a TV-prior for sparsity in the gradient domain (TV), Joint Total Variation (JTV) as well as a joint PLS reconstruction (PLS). 


\subsection{Phantoms and data sets}
For the numerical studies we set up an artificial brain phantom (see Figure \ref{fig:phantoms}) with a size of $272$x$272$ pixels using data from BrainWeb \cite{cocosco:brainweb}, which features the typical challenges of a joint reconstruction.
\begin{figure}[ht]
 \center
 \begin{tabular}{ccccc}
  \includegraphics[width=0.25\textwidth]{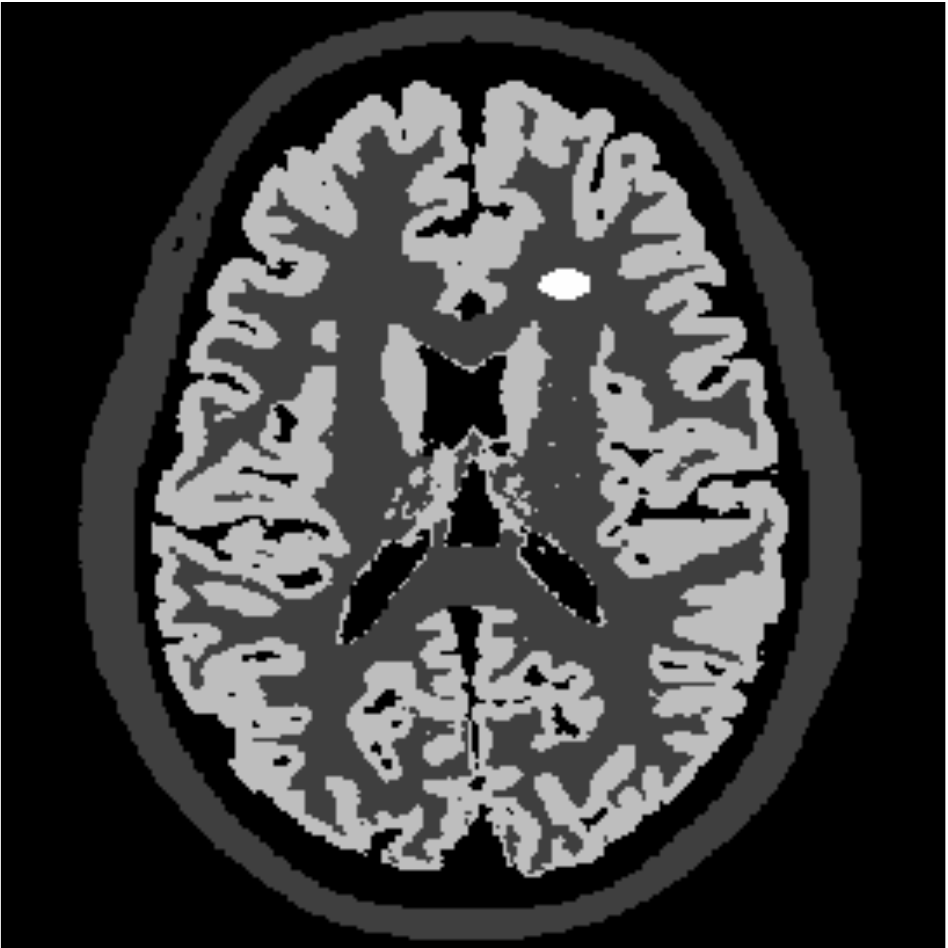} & \hspace{-1.3em} 
 \begin{overpic}[height=0.25\textwidth]{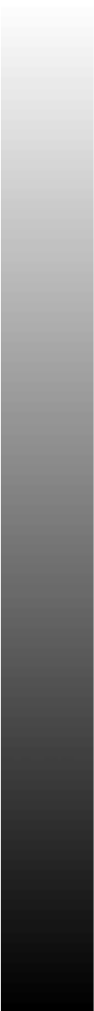}
  \put (0.75,91) {\small $10$} 
  \put (2.75,3) {\color{white}{$0$}}
 \end{overpic} & 
  \includegraphics[width=0.25\textwidth]{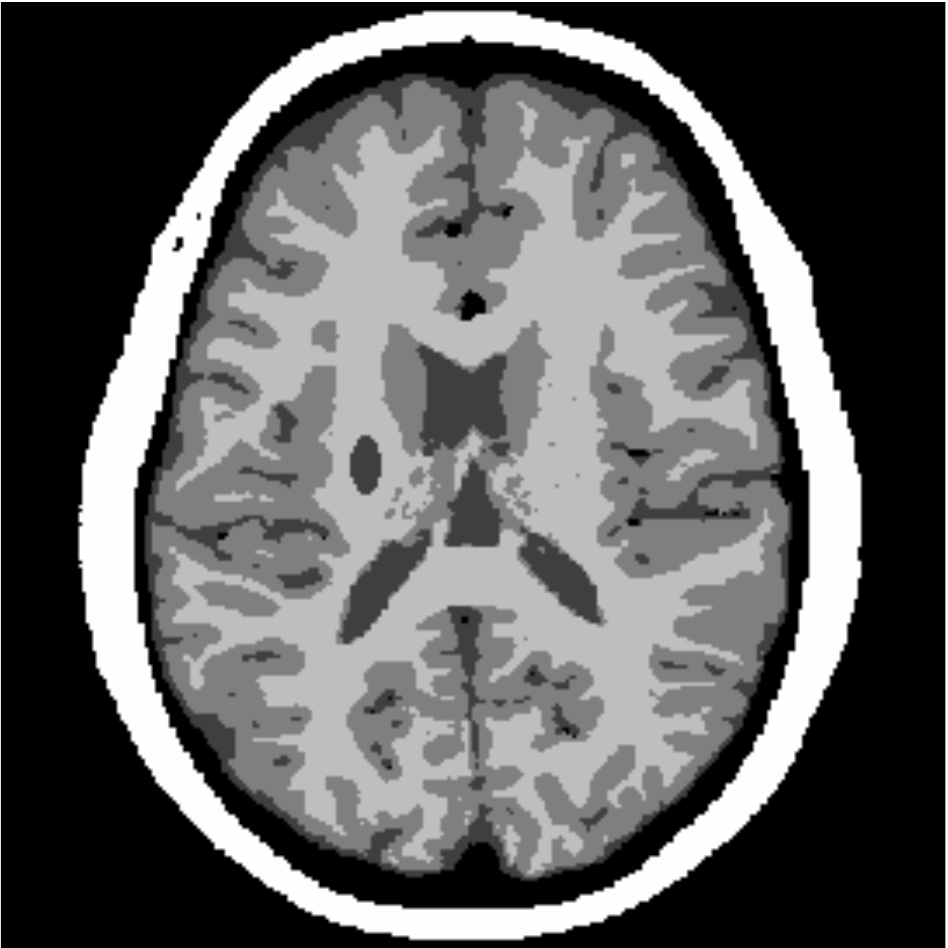} & \hspace{-1.3em} 
   \begin{overpic}[height=0.25\textwidth]{figure5a_cbar}
  \put (2.75,91) {\small $1$} 
  \put (2.75,3) {\color{white}{$0$}}
 \end{overpic} & 
  \includegraphics[width=0.25\textwidth]{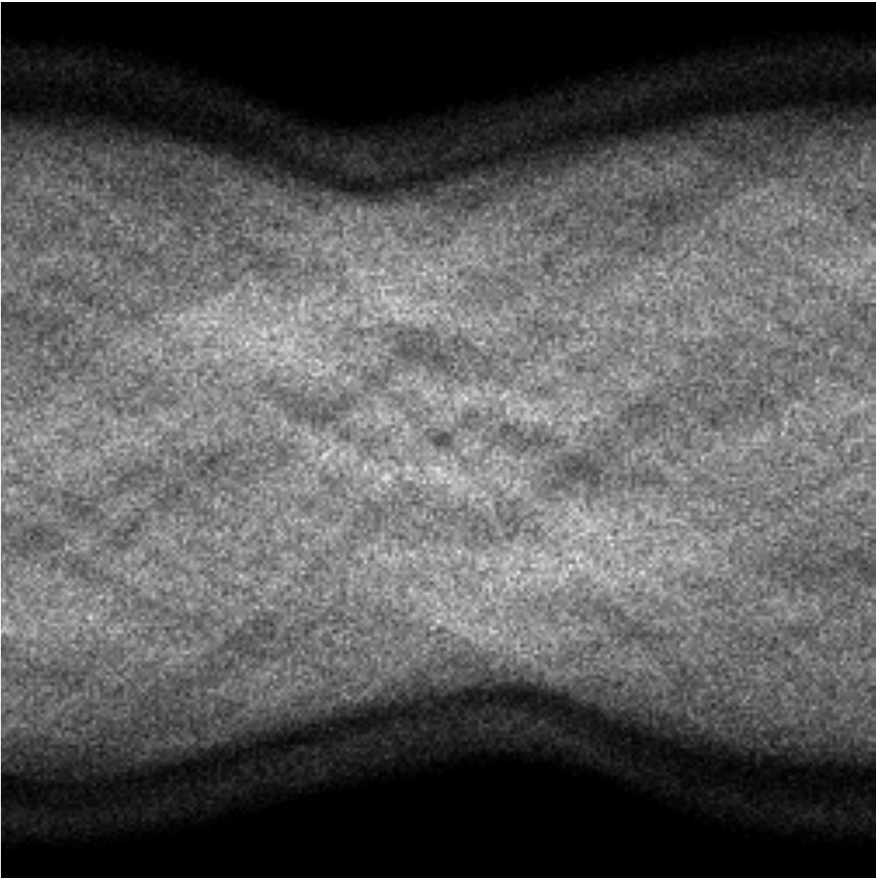} \\
  (a) PET phantom && (b) MRI phantom && (c) PET data
 \end{tabular}
 \caption{Phantoms and PET data.}
 \label{fig:phantoms}
\end{figure}
First of all, both phantoms feature the exact same location of edges everywhere except for the two artificially added lesions in the upper right part of the PET image and the left part of the MR image. 
Both have been added to illustrate a common joint reconstruction problem, namely the introduction of artifacts, i.e. features not present in one of the images that are transferred artificially to the other.  
While the locations of the edges are equal, the jumps across the edges are not, meaning that we e.g. have a ``step up'' from gray matter to white matter in the MR image, while we have a ``step down'' in the PET image. 
Furthermore, we point out the different range of image intensities. 
This imposes the issue of locally different gradient heights, which cannot be dealt with thoroughly by a global scaling of the data.
Summing up we find the following features: 
\begin{itemize}
 \item Equal edge locations (except for two lesions),
 \item Equal {\it and} different edge orientations,
 \item Different height of jumps across edges,
 \item Different scale.
\end{itemize}
The phantoms have been used to create artificial data sets for PET and MR imaging. 
The PET operator is taken from EMrecon \cite{koesters:EMrecon} and is modeled to match the geometry of one detector ring of the Siemens mMR scanner. 
It features 56 transversal detector blocks at 8 crystals each, and the gaps between detector blocks have been artificially filled by 56 virtual crystals. 
For further details we refer the reader to \cite{delso:performance}. 
In order to model the lack of resolution, positron range and scatter in PET, we employ a 2-dimensional Gaussian blur in the image domain with a full width at half maximum (FWHM) of $4$x$4$mm, where we assume that the pixel size in the image is $1$mm. 
The blurring kernel thus corresponds to a Gaussian with a standard deviation of $\sigma = 1.7$. 
The PET data is an instance of a Poisson distribution, where we simulated a total number of $1.5$ million counts. 
However, we point out that for artificial data the total number of counts is a rather poor criterion to judge a data set, since its quality rather depends on the distribution of those counts across the phantom. 

In case of MRI we use undersampled $k$-space data, meaning that we sample the Fourier space only at a few frequencies specified by different geometries. 
For the sake of simplicity, the MRI operator hence consists of a 2-dimensional Fourier transform followed by a projection onto the geometric pattern of the corresponding sampling (cf. \cite{Ehrhardt2016} . 
Note that we do not use a non-uniform Fourier transform since the chosen frequencies are still located on a Cartesian grid.
However, the idea and method do not change for non-Cartesian methods.
Eventually, Gaussian noise with an energy of approximately five percent of the total energy of the data set is added. 
We show four different types of undersampling in this work, which can be seen in Figure \ref{fig:samplings}.
The samplings introduce different types of artifacts and hence serve different purposes for the joint reconstruction setting which we elaborate on alongside with the results below.

\begin{figure}[ht!]
 \begin{center}
 \begin{tabular}{cccc}
  \includegraphics[width=0.22\textwidth]{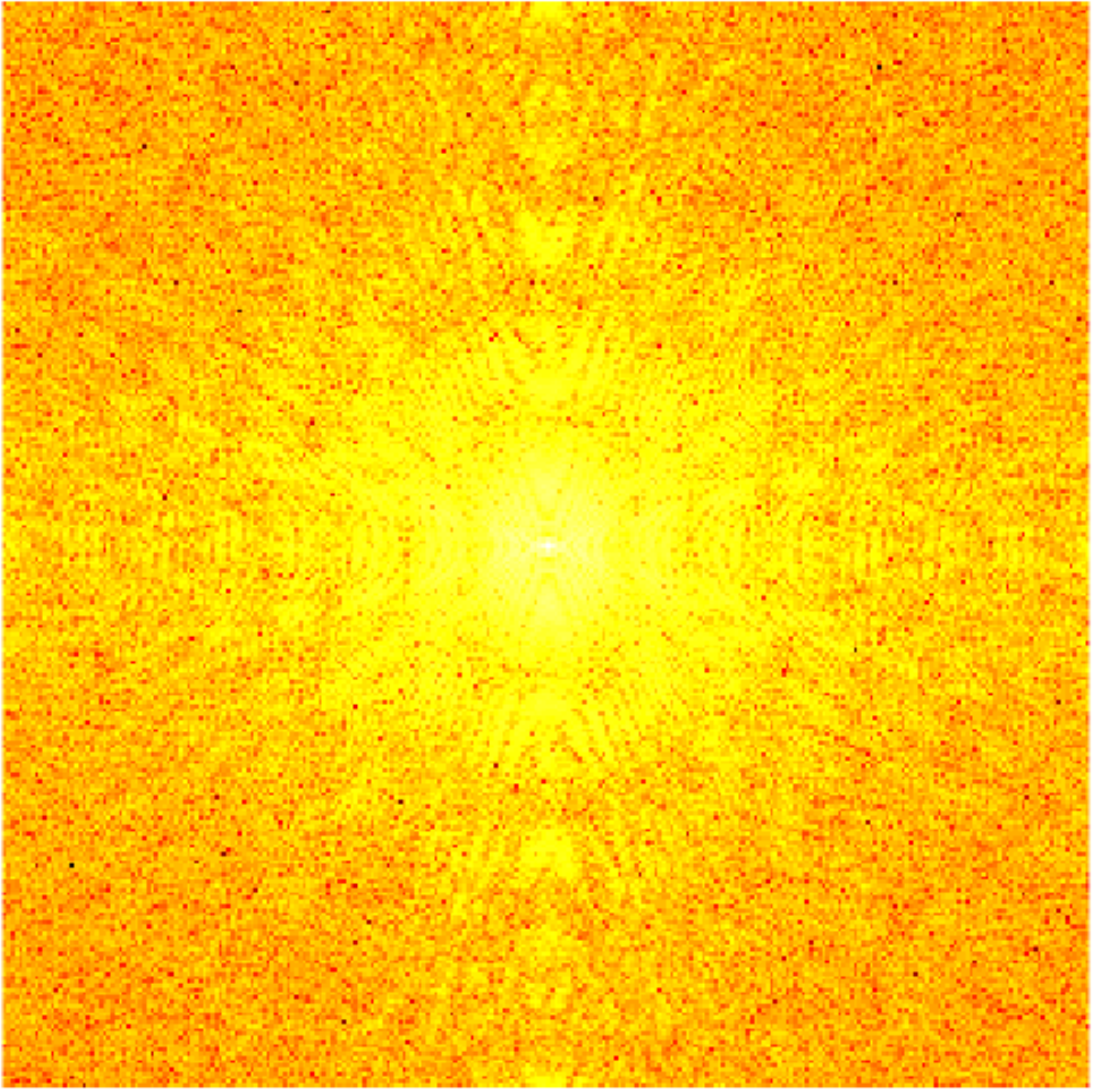} &  
  \includegraphics[width=0.22\textwidth]{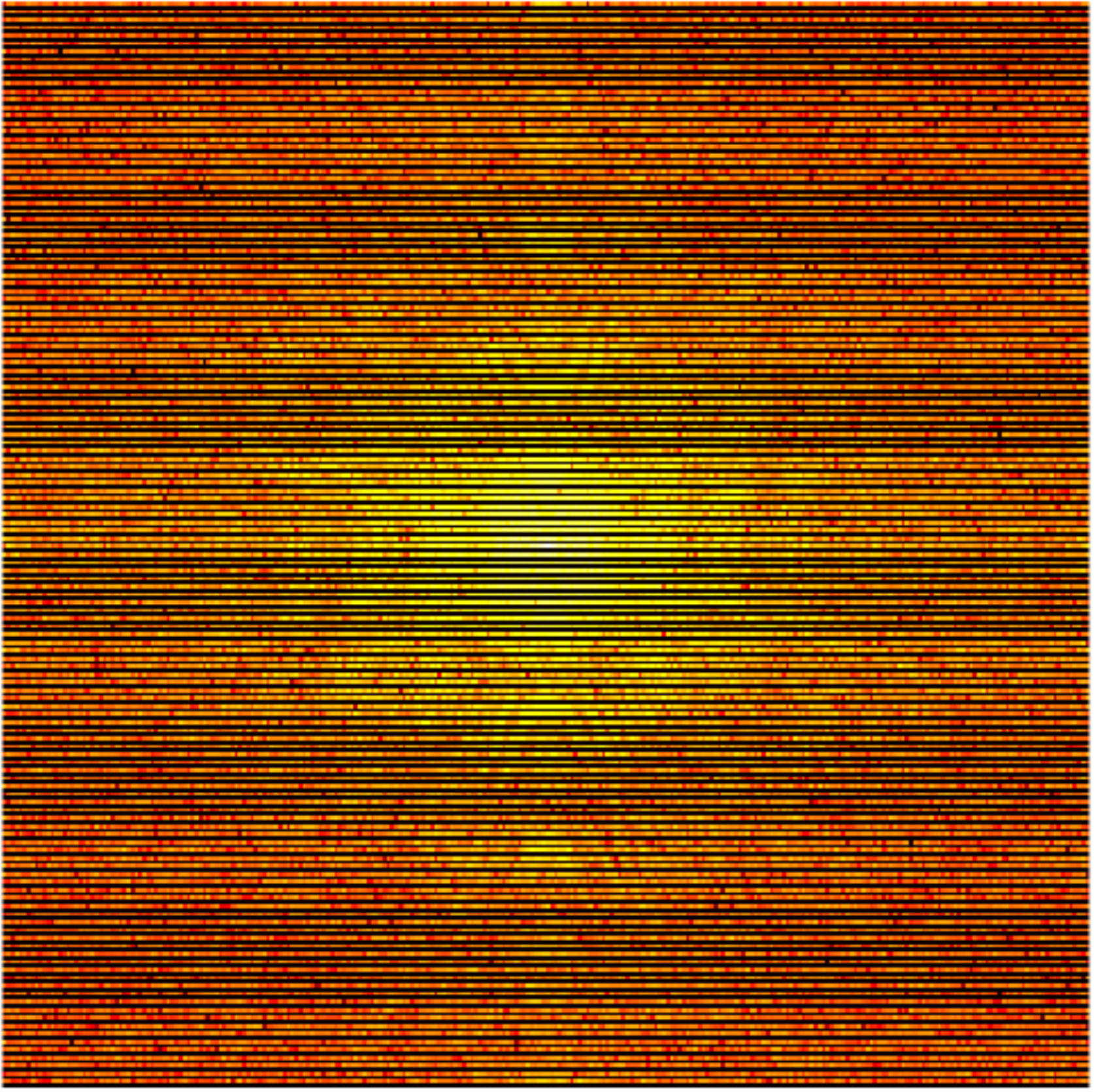} &  
  \includegraphics[width=0.22\textwidth]{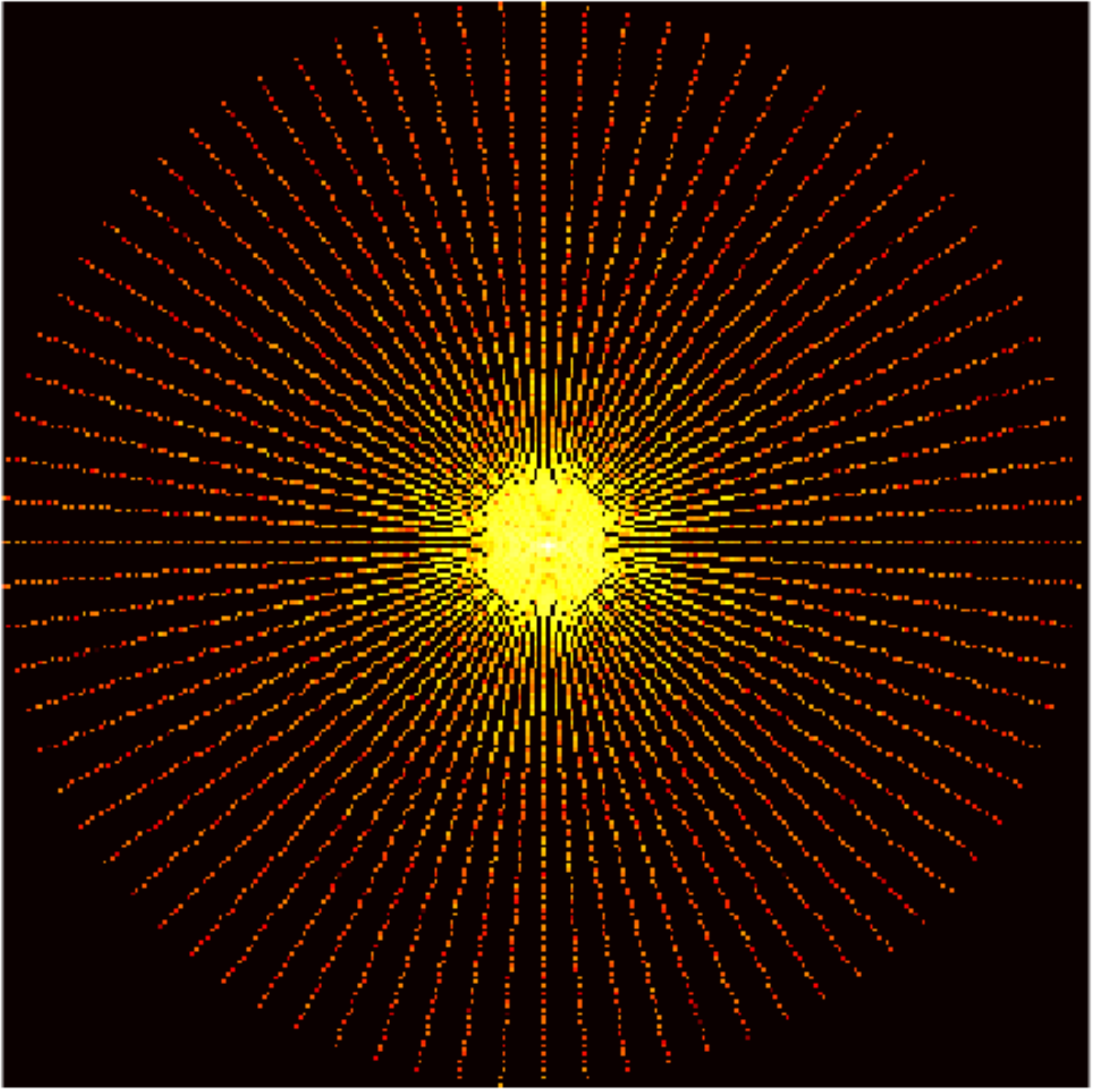} &
  \includegraphics[width=0.22\textwidth]{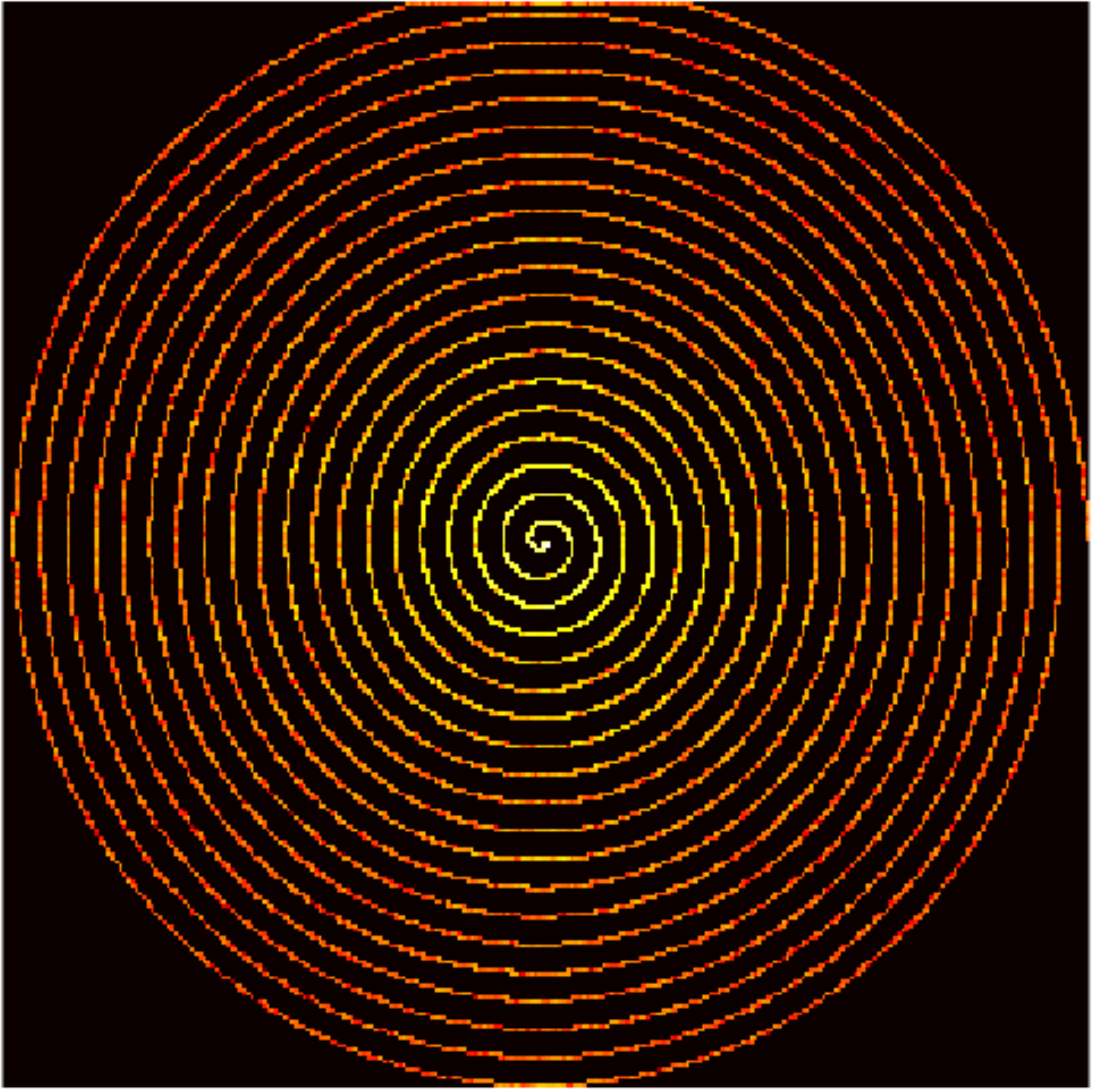} \\
  (a) Full: $100 \%$ & (b) Half: $10 \%$ & (c) Spokes: $50 \%$ & (d) Spiral: $13 \%$
 \end{tabular}
 \caption{Discrete sampling geometries for MRI with percentage of sampled Fourier coefficients: Full sampling (a), every second line (b), spokes through the $k$-space center (c) and a uniform spiral (d).}
  \label{fig:samplings}
 \end{center}
\end{figure}

\subsection{Methods}
We compare our method (ICB)
\begin{alignat}{5}
u^{k+1} \in \arg \min_{ u \geq 0} &\Big\{ \alpha \DKL(f,Au) && + \lambda D_{\TV}^{p_u^k}(u, u^k) && + (1-\lambda) \ICB^{p_v^k}(u,v^k) && \Big\}, \label{eq:PETMRIJointReconstruction1} \\
v^{k+1} \in \arg \min_{v \geq 0} &\Big\{ \beta \|Bv-g\|_\C^2 && + \mu D_{\TV}^{p_v^k}(v, v^k) && +  (1-\mu) \ICB^{p_u^k}(v,u^k) && \Big\}, \label{eq:PETMRIJointReconstruction2} 
\end{alignat}
to several reconstruction techniques:
separate reconstructions with no prior (Expectation Maximization \cite{Shepp:MaximumLikelihoodReconstruction} with early stopping for PET and zero-filled inverse FT for MRI) and a total variation prior, 
as well as joint reconstructions via joint TV (JTV \cite{Haber:ModelFusion}) and linear parallel level sets (LPLS \cite{Ehrhardt:JointReconstruction}).
Since decoupled Bregman iterations did not provide a substantial improvement over TV, we did not include these results. 
For JTV and LPLS we used the implementation provided by the authors of \cite{Ehrhardt:JointReconstruction}. 
We mention that due to the nonconvexity of the LPLS method we follow \cite{Ehrhardt:JointReconstruction} and use the result of JTV as the initialization for quadratic parallel level sets (QPLS \cite{Ehrhardt:JointReconstruction}), and the result of QPLS to initialize LPLS.
Since the QPLS reconstructions have always been inferior to LPLS we left them out of this paper as well.
Hence, for separate PET reconstructions we solve 
\begin{align}
 &\min_{u \geq 0} ~ \DKL(f,Au), \tag{no prior} \\
 &\min_{u \geq 0} ~ \alpha \DKL(f,Au) + \TV(u), \tag{TV}
\end{align}
and accordingly for MRI 
\begin{align}
 &\min_{u \geq 0} ~ \frac{1}{2} \|Bv -g \|_{\C}^2, \tag{no prior} \\
 &\min_{u \geq 0} ~ \frac{\beta}{2} \|Bv -g \|_{\C}^2 + \TV(u). \tag{TV}
\end{align}
The joint reconstruction problems read 
\begin{align}
 &\min_{u,v \geq 0} ~ \alpha \DKL(f,Au) + \frac{\beta}{2} \|Bv -g \|_{\C}^2 + \text{JTV}(u,v) \tag{JTV} \\ 
 &\min_{u,v \geq 0} ~ \alpha \DKL(f,Au) + \frac{\beta}{2} \|Bv -g \|_{\C}^2 + \text{LPLS}(u,v), \tag{LPLS} 
\end{align}
where 
\begin{align*}
 \text{JTV}(u,v) = \int_{\Omega} \sqrt{|\nabla u|^2 + |\nabla v|^2} \dx, \qquad 
 \text{LPLS}(u,v) = \int_{\Omega} |\nabla u||\nabla v| - | \nabla u \cdot \nabla v | \dx. 
\end{align*}

All parameters were tuned to maximize the self-similarity index (SSIM \cite{WangSSIM}) between the reconstructions and the ground truth over the region of interest, where for joint reconstruction methods we maximized the `joint' SSIM, i.e. the sum of both SSIM values.
For all joint reconstruction techniques, the ratio between the parameters $\alpha$ and $\beta$ has been chosen such that the size of the data terms is approximately equal.  
This guarantees an equal regularization for both the PET and the MR image, which corresponds to a uniform evolution of both the PET and MR image for ICB in terms of regularity during the Bregman iterations. 
We mention that there is potential further improvement by tweaking this ratio for data of different quality.
The Bregman iterations for ICB have been stopped when noise reappeared in the images. 
Starting with high regularization, we performed 50 Bregman iterations on average.
With the numerical implementation outlined in Section \ref{sec:numerics}, $tol_{gap} = 1e-4$ and $tol_{constraint} = 1e-5$, every Bregman iteration took approximately three minutes on a standard laptop, which amounts to a total computation time of approximately three hours.
We however remark that it is possible to reduce both the amount of Bregman iterations and the tolerances without substantially reducing the quality of the results.
The SSIM values can be found in Table \ref{tab:SSIM} alongside with the weighting between channels for ICB.
\begin{table}[t!]
\center
\begin{tabular}{rcccccccc}
 & \multicolumn {2}{c}{{\bf Full}} & \multicolumn {2}{c}{ {\bf Half} } & \multicolumn {2}{c}{ {\bf Spokes} } & \multicolumn {2}{c}{ {\bf Spiral} } \\
 \hline \hline
 & PET & MRI & PET & MRI & PET & MRI & PET & MRI \\
 \hline
 no prior & $0.5366$ & $0.9940$ & $0.5366$ & $0.6722$ & $0.5366$ & $0.7804$ & $0.5366$ & $0.5283$ \\
 TV       & $0.5657$ & $0.9942$ & $0.5657$ & $0.6718$ & $0.5657$ & $0.9415$ & $0.5657$ & $0.7807$ \\
 \hline
 JTV      & $0.6806$ & $0.9944$ & $0.6538$ & $0.8194$ & $0.6340$ & $0.8977$ & $0.6540$ & $\mathbf{0.9101}$ \\
 LPLS     & $0.7837$ & $\mathbf{0.9949}$ & $0.6812$ & $0.8322$ & $0.6485$ & $0.8541$ & $0.6485$ & $0.8537$ \\
 \rowcolor[gray]{.8}
 \textbf{ICB}  & $\mathbf{0.9218}$ & $\mathbf{0.9999}$ & $\mathbf{0.8025}$ & $\mathbf{0.8969}$ & $\mathbf{0.7695}$ & $\mathbf{0.9453}$ & $\mathbf{0.7017}$ & $0.8240$\\
 \hline \hline
 $\lambda$ / $\mu$ & $0.1$ & $1$ & $0.3$ & $0.8$ & $0.1$ & $1$ & $0.2$ & $0.6$
\end{tabular}
\caption{SSIM values for single and joint reconstruction results for different samplings. 
The parameters $\lambda$ and $\mu$ are the weighting of the own channel during the joint reconstruction with ICB. 
E.g. for 'half', the influence of MRI during the PET reconstruction was 70\%, the influence of PET on the MRI reconstruction was 20\%.}
\label{tab:SSIM}
\end{table}

\subsection{Results}

Figures \ref{fig:res_full} to \ref{fig:res_spiral} show the results for four different MRI samplings and different types of regularization. 
The left two columns display the PET and MRI reconstructions, where all pictures are put onto the original scale of the ground truth, i.e. $[0,10]$ for PET and $[0,1]$ for MRI. 
If the reconstructions overestimate the image values beyond that scale, the corresponding pixels are set to the maximum value (note that an underestimation below zero is impossible due to the positivity constraint).
The effective over- or underestimation of image values with respect to the ground truth can be assessed from the difference images on the right-hand side of the figures.
Note that due to the missing weighting between gradients for JTV and LPLS, we had to rescale the PET data by a factor of $10$ in order to approximately provide the same range of image intensities for both PET and MRI. 
The results have then eventually been rescaled to their original range. 
We mention that due to the nonlinearity of the methods this may result in a slight change of quantitative values.

\begin{figure}[t!]
\center
 \includegraphics[width=0.95\textwidth]{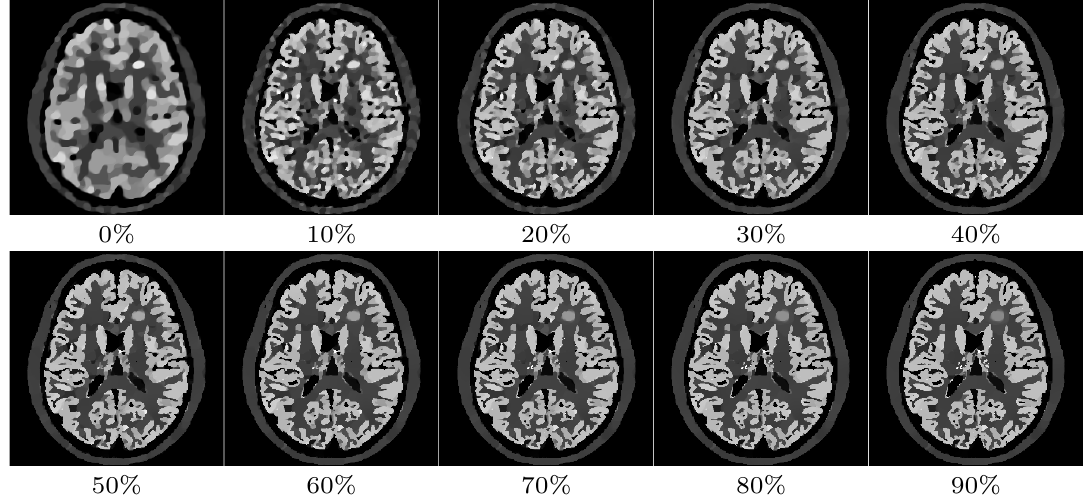}
 \caption{PET channel weighting for full MRI data: PET reconstruction for increasing amount of influence of the MR channel from 
 no influence (upper left) to 90 percent influence (lower right).}
 \label{fig:weighting_smp1}
\end{figure}

\subsubsection{Full sampling}
For a full sampling, already a separate reconstruction provides a visually perfect MR image, which remains the case for all priors.
Here, the weighting for ICB is chosen such that the PET image does not influence the reconstruction of the MR image ($\mu = 1$) and the procedure is similar to a PET reconstruction with an (evolving) anatomical prior.
The MR reconstruction performed in parallel hence corresponds to a single channel BTV reconstruction.
The quality of the PET image varies greatly. 
Without any prior, the PET image shows the typical noisy and blurry appearance of reconstructions from noisy Poisson data. 
The TV prior removes the noise, but highly oversmoothes the image. 
The first joint reconstruction method, JTV, is able to transfer some part of the sharp structures of the MR image to the PET image and increases its quality. 
However, the result remains too smooth. 
LPLS and ICB both show a substantial improvement of the image quality in terms of sharp edges and noise reduction. 
However, LPLS still features some remaining noise, and the lesion in the MRI not present in PET has been partly transferred. 
The result for ICB seems visually perfect on the shared structures of the image which can as well be assessed from the difference image.
The only drawback is the smoothing of the hot lesion only present in the PET image, while we however observe that the MRI lesion has not been transferred.
The observations are as well confirmed by the SSIM values in Table \ref{tab:SSIM}.

The advantage of ICB over LPLS is easily explained by the weighting between channels. 
While LPLS always relies on a 50/50 weighting, the result for ICB has been obtained with an MR influence of $90 \%$. 
For further comparison to a 50/50 weighting (and the overall influence of the channel weighting) we provide a weighting series in Figure \ref{fig:weighting_smp1}.

\subsubsection{Half sampling}
Sampling every second line of the k-space introduces a ghosting artifact into the MR image. 
JTV and LPLS show a similar performance in removing some parts of these artifacts from the MRI, 
while LPLS as well shows a clear improvement of the PET image.
However, both lesions are again artificially shared between the images. 
The ICB results contain the least artifacts for MRI and the highest improvement for the PET image, 
where however again the hot lesion is attenuated. 
We as well find a slight transfer of the lesion from the MR image into the PET image, whereas the lesion from PET is not transferred to the MR image.

\subsubsection{Spokes sampling}
The situation for a spokes sampling is almost identical to a full MR sampling. 
Since already a TV prior removes the resulting grain artifacts, we again choose the weighting for ICB such that the PET image does not influence the MR image ($\mu = 1$). 
Interestingly, JTV and LPLS decrease the quality of the MR image, both visually and in terms of SSIM, by re-introducing some part of the noise. 
However, the PET image still benefits significantly from the influence of the MRI. 
ICB again shows the best performance for PET, sharpening the edges and restoring the quantitative values, where we again observe the attenuation of the hot lesion and transfer of the MR lesion. 

\begin{figure}[t!]
\center
 \includegraphics[width=0.95\textwidth]{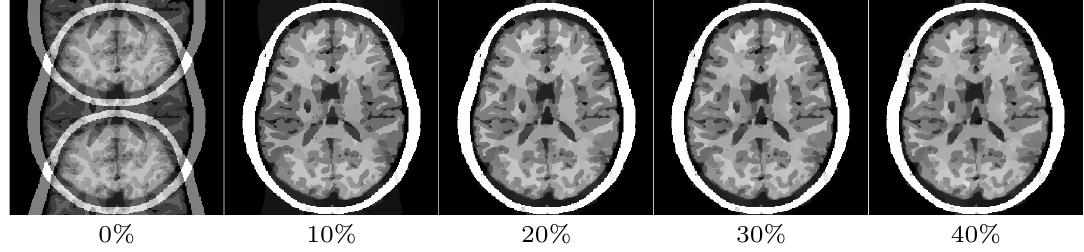}
 \caption{MRI channel weighting: MRI reconstruction for increasing amount of influence of the PET channel from 
 no influence (left) to 40 percent influence (right).}
 \label{fig:weighting_smp2}
\end{figure}

\subsubsection{Spiral sampling}
For a spiral sampling, we again expect a mutual benefit for both modalities, since the sampling introduces some circular artifacts. 
Surprisingly, JTV delivers the best result in terms of SSIM, which however look too smooth. 
Visually the results for LPLS convince the most, both for PET and MRI. 
The ICB result for PET is sharper and more accurate in terms of quantitative values, but appears a little too noisy, which can as well be caused by the influence of the undersampled MR image.

\subsection{The benefit of channel weighting}
We further comment on the channel weighting, i.e. the parameters $\lambda$ and $\mu$. 
As already mentioned in the previous section, a known drawback of joint reconstruction methods is a cross-talk between the different channels, i.e. 
channels might cause artifacts in other channels. 
This can e.g. be the transfer of a structure from one channel to the other
(cf. the PET image in Fig. \ref{fig:res_spokes}), or the smoothing out of structures since they are
not present in all channels (cf. the lesion in the PET image in Fig. \ref{fig:res_full}).
This is an issue which is hard to avoid. 
However, the weighting between channels, i.e. the parameters $\lambda$ and $\mu$ can indicate, which structures 
are likely to be artifacts or falsely influenced by the joint reconstruction. 
Figure \ref{fig:weighting_smp1} shows the PET reconstruction with an increasing amount of influence of the MR channel. 
In comparison to a seperate reconstruction (upper left), already an influence of 10 percent of the MR image substantially sharpens
the PET image. 
This trend continues as the influence further increases. 
Unfortunately, the hot lesion in the upper right part of the PET image starts to be smoothed out with increasing MR weighting. 
This error, however, can be identified by a look at the seperate reconstruction and low MR weightings. 

This of course works in both ways. 
Figure \ref{fig:weighting_smp2} shows the MR image for the `half' sampling under increasing weighting of the PET image. 
Already an influence of 10 percent substantially reduces the ghosting artifacts, which however cannot be suppressed entirely 
even for higher weighting. 

Hence, a series of reconstructions with different channel weightings can help identify artifacts caused by the joint reconstruction.

\begin{figure}[t]
 \center
 \includegraphics[width=0.8\textwidth]{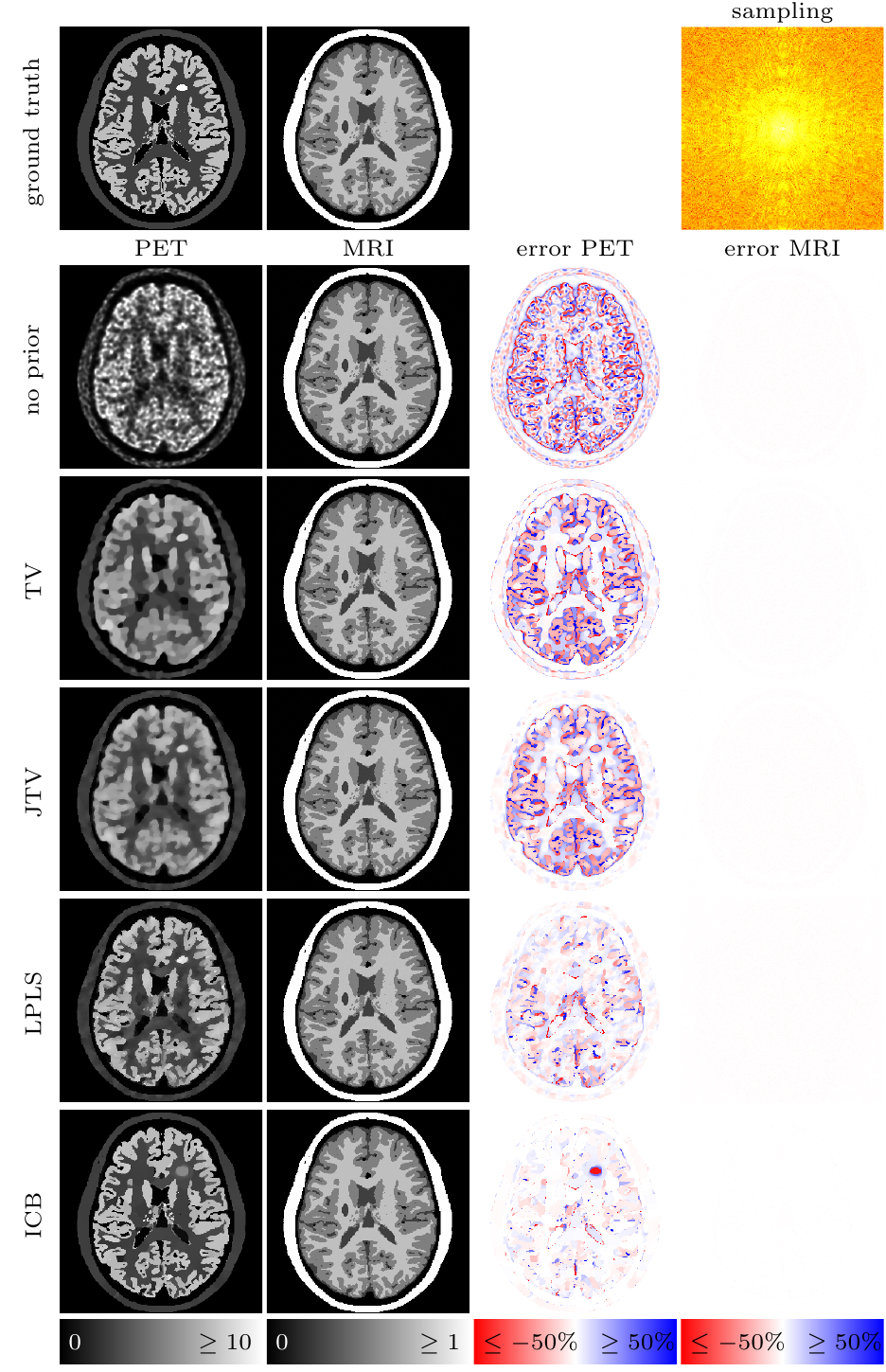}
 \caption{Reconstruction results for a full sampling of the MRI and different priors. 
 Left-hand side: reconstructions. Right-hand side: difference to the ground truth.}
 \label{fig:res_full}
\end{figure}

\begin{figure}[ht] 
 \center
 \includegraphics[width=0.8\textwidth]{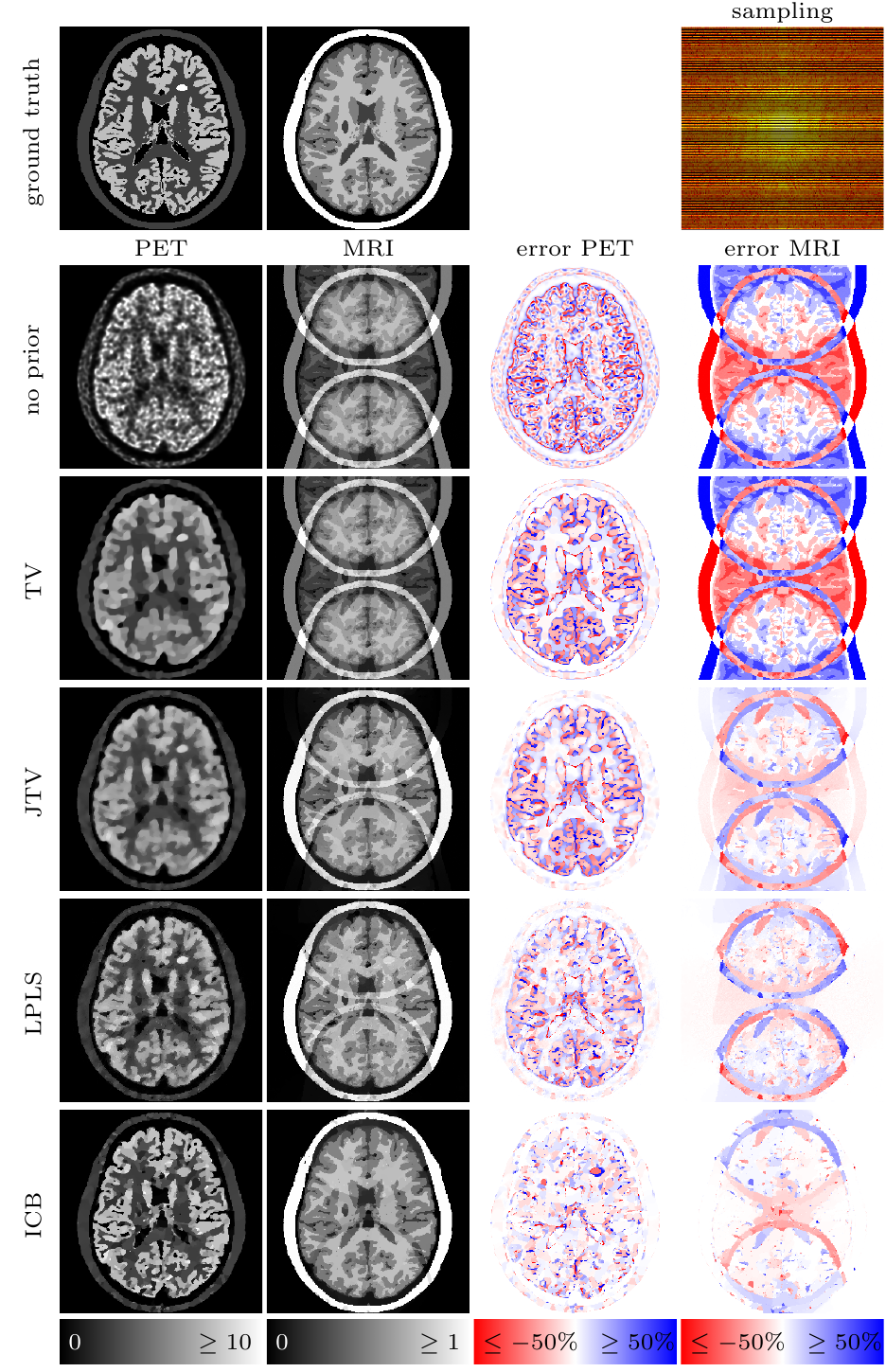}
 \caption{Reconstruction results for sampling every second line and different priors. 
 Left-hand side: reconstructions. Right-hand side: difference to the ground truth.}
 \label{fig:res_half}
\end{figure}

\begin{figure}[ht]  
 \center
 \includegraphics[width=0.8\textwidth]{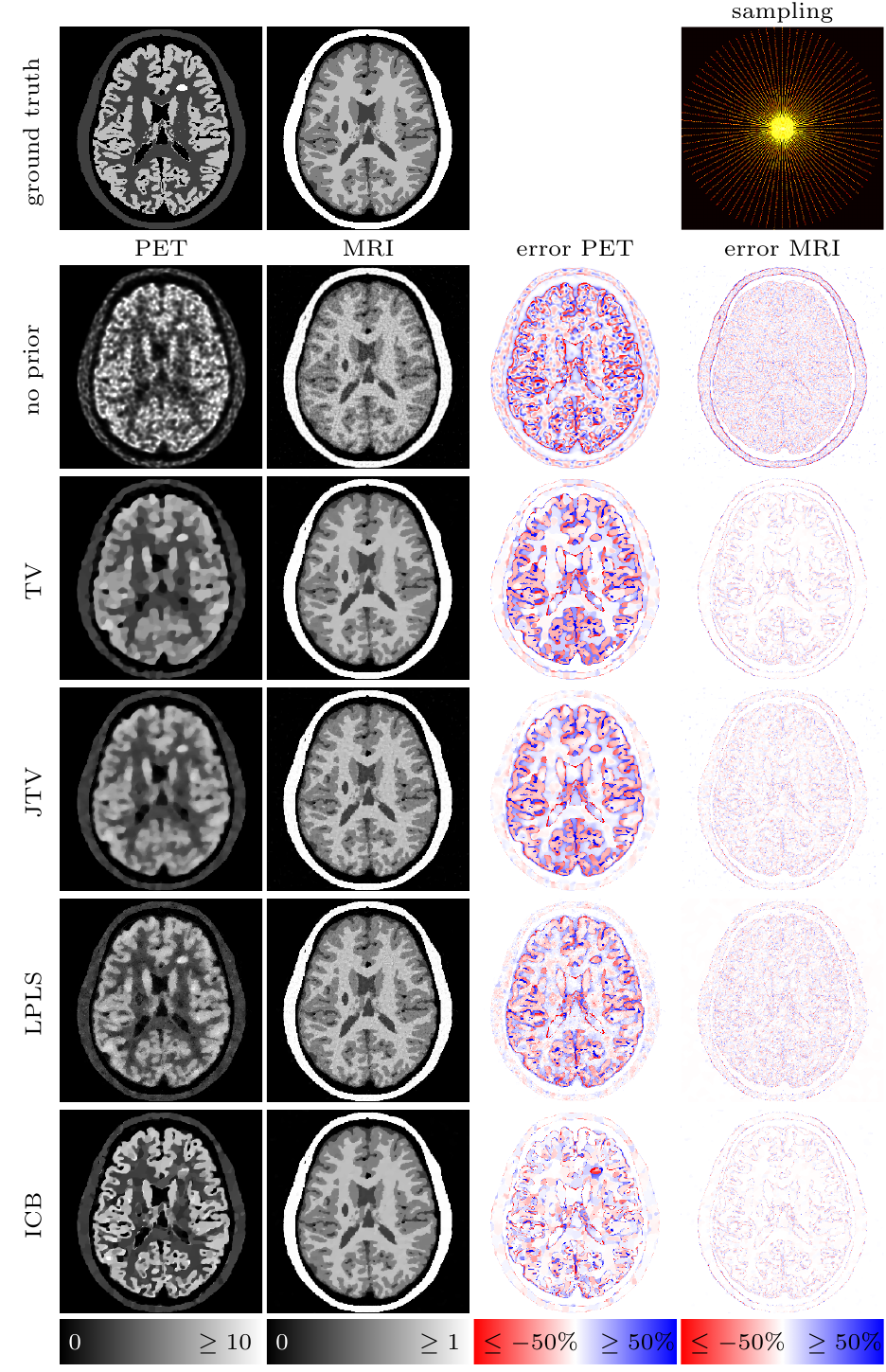}
 \caption{Reconstruction results for a spokes sampling of the MRI and different priors. 
 Left-hand side: reconstructions. Right-hand side: difference to the ground truth.}
 \label{fig:res_spokes}
\end{figure}

\begin{figure}[ht] 
 \center
 \includegraphics[width=0.8\textwidth]{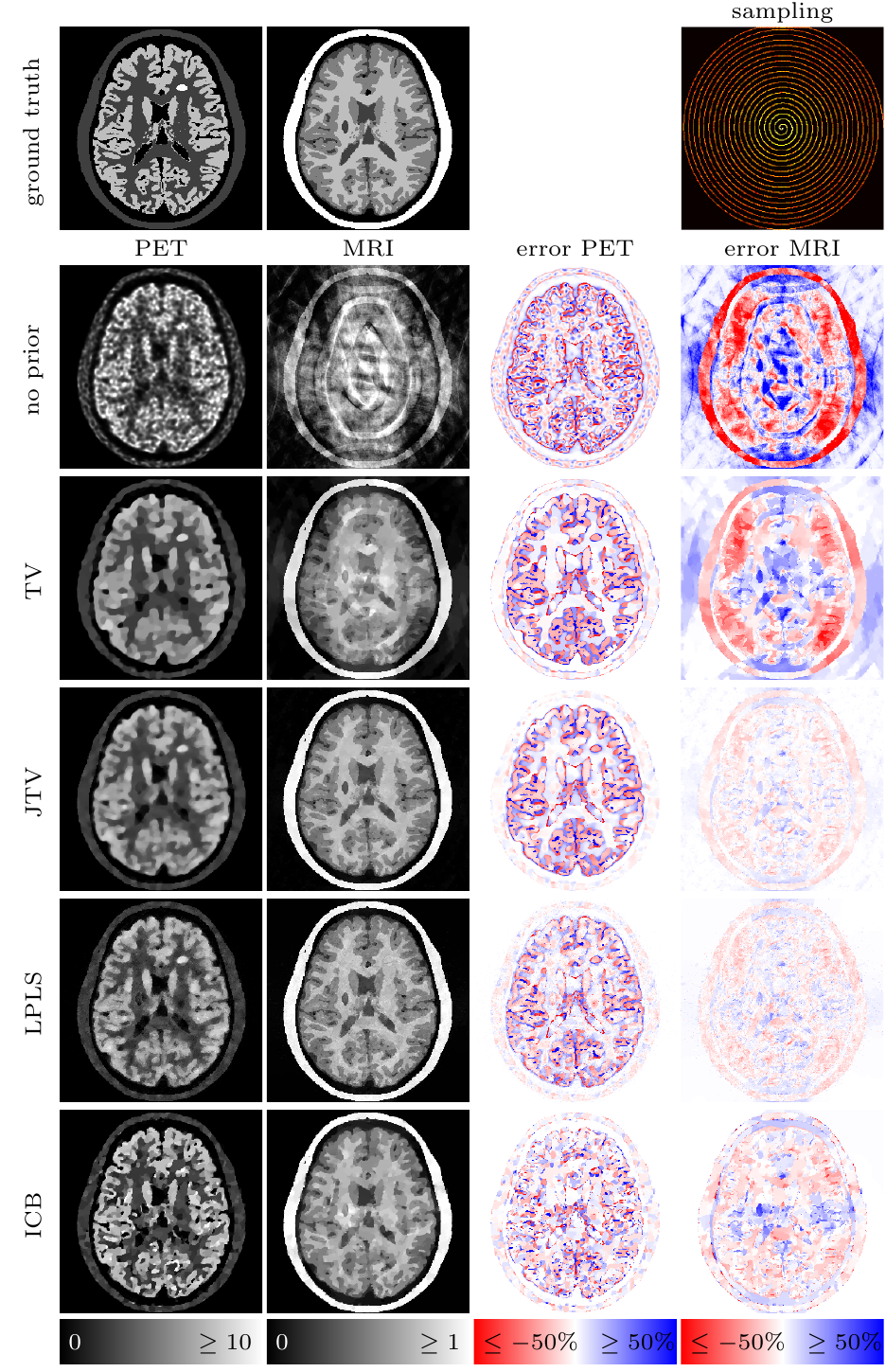}
 \caption{Reconstruction results for a spiral sampling of the MRI and different priors. 
 Left-hand side: reconstructions. Right-hand side: difference to the ground truth.}
 \label{fig:res_spiral}
\end{figure}

%% file: conclusion.tex
\section{Conclusion}\label{sec:conclusion}

We extended the ideas of Moeller et al. \cite{Moeller:ColorBregmanTV} for color image processing to a more general joint reconstruction setting, 
involving data and hence images of totally different types.
The presented method is able to exploit structural similarities between images during the reconstruction to obtain significant improvements even for highly degraded data. 
We gave insights and justification for the method in a rigorous, theoretical setting and extended the idea from a discrete vector space setting to the space of finite Radon measures.  
 
Using PET-MR imaging as a particular numerical example, we demonstrated that both modalities can substantially benefit from a joint reconstruction. 
We furthermore compared the method to already existing joint reconstruction methods and showed a similar or in many cases superior performance. 

There are a few questions remaining open for the future. 
First of all, a current subject of investigation is the choice of subgradients, which still leaves room for improvements. 
This could in particular solve the missing existence proof for the infimal convolution method. 
Regarding applications, an investigation of the method's performance on real data is certainly interesting, 
as well as the transfer to other fields such as functional MRI, which is subject of current research.
Last but not least, we are investigating other numerical approaches to reduce the computational effort for the solution of the method.

%% file: appendix.tex
\begin{appendix} 
\section{Appendix}
\subsection{Proof of Theorem \ref{thm:inf_conv}}
\begin{mythm}
Let $\nu,\eta \in \M$ such that $\mathcal{C}(\nu) := \partial \|\nu\nM \cap \CO \neq \emptyset$.  
Let $q \in \mathcal{C}(\nu)$. Then
\begin{align*}
\left[D_{\| \cdot \nM}^q(\cdot,\nu) \Box D_{\| \cdot \nM}^{-q}(\cdot,-\nu)\right](\eta) 
= \int_\Omega G(f_\eta,q) \di |\eta| , 
\end{align*}
where 
\begin{align}
G(f_\eta,q) = \begin{cases}
	       1 - |\cos(\varphi)| |q|, & \text{ if } |q| < |\cos(\varphi)|,\\
	       | \sin(\varphi)| \sqrt{1 - |q|^2}, &\text{ if } |q| \geq |\cos(\varphi)|
	       \end{cases}
\end{align}
$|\eta|$-a.e.,
 and $\varphi$ denotes the angle between $f_\eta$ and $q$, i.e. 
 $\cos(\varphi) |f_\eta| |q| = f_\eta \cdot q$, 
 with $\varphi(x) := 0$ if $q(x) = 0$.
\end{mythm}
\begin{proof}
We can adapt the idea of the proof in \cite{2016debiasing}. Let 
\begin{alignat*}{2}
&f_1(\eta) = D_{\| \cdot \nM}^q(\eta,\nu) &&= \|\eta \nM- \langle q,\eta \rangle, \\
&f_2(\eta) = D_{\| \cdot \nM}^{-q}(\eta,-\nu) &&= \|\eta \nM + \langle q, \eta \rangle.
\end{alignat*}
It follows from $(f_1 \Box f_2)^* = f_1^* + f_2^*$ and the definition of the biconjugate that
\begin{align*}
 f_1 \Box f_2 \geq (f_1^* + f_2^*)^*.
\end{align*}
We start with computing the right-hand side and find $f_1^*(w) = \iota_{\mathcal{M}^*}(w + q)$ and $f_2^*(w) = \iota_{\mathcal{M}^*}(w -q)$,
where $\iota_{\mathcal{M}^*}$ denotes the characteristic function of the dual ball
\begin{align*}
 B = \big\{ w \in \M^* ~|~ \| w \|_{\mathcal{M}^*} \leq 1 \big\}.
\end{align*}
Since $\CO \subset \M^*$, we deduce
\begin{alignat*}{3}
&&(f_1^* + f_2^*)^*(\eta) &= \sup_{w \in \M^*} \langle \eta, w \rangle  &&  \text{ s.t.} ~ \|w+q\|_{\mathcal{M}^*} \leq 1, \|w-q\|_{\mathcal{M}^*} \leq 1 \\
&&&\geq \sup_{w \in \CO } \int_\Omega w \cdot f_\eta \di |\eta| \quad &&\text{ s.t.} ~ \|w+q\|_\infty \leq 1, \|w-q\|_\infty \leq 1. 
\end{alignat*}
Due to the constraints, we can carry out the computation pointwise (respectively $|\eta|$-a.e.). 
We first address the case $|q| = 1$, which immediately implies that $w = 0$ and $w \cdot f_\eta = 0$.
Hence we can from now on assume that $|q| < 1$ and set up the corresponding Lagrangian for $|\eta|$-a.e. $x \in \Omega$ 
\begin{align}
\mathcal{L}(w(x), \lambda_1, \lambda_2) = - w(x) \cdot f_\eta(x) + \lambda_1 (|w(x) - q(x)|^2 -1) + \lambda_2 (|w(x) + q(x)|^2 -1),
\label{eq:Lagrange}
\end{align}
where we leave out the dependence on $x \in \Omega$ for simplicity.
Note that $w,q,f_\eta \in \R^m$ and that we \textit{minimize} the Lagrangian here. 
Every optimal point of (\ref{eq:Lagrange}) has to fulfill the four Karush-Kuhn-Tucker conditions, namely
\begin{align*}
& \frac{\partial \mathcal{L}}{\partial w}(w, \lambda_1, \lambda_2) = 0, \quad &\lambda_1 (|w -q|^2-1) = 0, \\ 
&\lambda_1, \lambda_2 \geq 0, & \lambda_2 (|w + q|^2 -1) = 0,
\end{align*}
and Slater's condition implies the existence of Lagrange multipliers for a KKT-point of (\ref{eq:Lagrange}). 
The first KKT-condition yields
\begin{align}
- f_\eta + 2 \lambda_1 (w-q) + 2 \lambda_2 (w+q) = 0.
\label{eq:FirstOrderOptimality}
\end{align}
We note that the case $f_\eta = 0$ (which can happen on $|\eta|$-zero sets) causes the objective functional to vanish, hence in the following $f_\eta \neq 0$.

Let us first consider the case $q=0$ in which \eqref{eq:FirstOrderOptimality} yields $f_\eta = 2 (\lambda_1 + \lambda_2) w$.
In case $|w| < 1$, we find that $(\lambda_1 + \lambda_2) = 0$ and hence $f_\eta = 0$, which is a contradiction. 
In case $|w| = 1$, we obtain that $(\lambda_1 + \lambda_2) = \frac{1}{2}$ and $w = f_\eta $ from \eqref{eq:FirstOrderOptimality}, hence 
\begin{align*}
 w \cdot f_\eta = |f_\eta|^2 = 1.
\end{align*}
If $q \neq 0$, we have to distinguish three cases:\\
\textbf{1st case:} $|w - q| <1, |w + q| = 1$. \\
Thus $\lambda_1 = 0$ and (\ref{eq:FirstOrderOptimality}) yields $f_\eta = 2 \lambda_2 (w + q)$. 
Since $|w + q| = 1$ and $|f_\eta| = 1$ we deduce $\lambda_2 = \frac{1}{2}$, so $w = f_\eta - q$ 
and finally for the value of the objective function
\begin{align*}
w \cdot f_\eta = |f_\eta|^2 - f_\eta \cdot q = 1 - f_\eta \cdot q.
\end{align*}
\textbf{2nd case:} $|w +q| <1, |w -q| = 1$ \\ 
We analogously find 
\begin{align*}
w \cdot f_\eta = 1 + f_\eta \cdot q.
\end{align*}
The first two cases thus occur whenever (insert $w$ in the conditions) $\left| f_\eta \pm 2q \right| < 1$, meaning that either of the conditions has to be fulfilled. 
We calculate
\begin{align}
| f_\eta - 2q |^2 < 1 \Leftrightarrow |q|^2 < q \cdot f_\eta \Leftrightarrow |q| < \cos(\varphi).
\end{align}
Hence $q\cdot f_\eta >0$ and 
\begin{align*}
1 - q \cdot f_\eta = 1 -|q \cdot f_\eta|. 
\end{align*}
In the second case we analogously find $|q| < -\cos(\varphi)$, hence $q \cdot f_\eta <0$ and
\begin{align*}
1 + q \cdot f_\eta = 1 -|q \cdot f_\eta|, 
\end{align*}
so we may summarize the first two cases as 
\begin{align*}
w \cdot f_\eta = 1 - |q \cdot f_\eta| =  1 - |\cos(\varphi)| |q|,
\end{align*}
if $|q| < |\cos(\varphi)|$. \\

\noindent
\textbf{3rd case:} $|w -q | =1, |w +q| = 1$ \\
At first we observe that from $|w + q |^2 = |w - q|^2$ we may deduce that $w \cdot q = 0$. 
Therefore we have $|w +q| = 1 \Rightarrow |w| = \sqrt{1-|q|^2}$. 
In the third case the optimality condition (\ref{eq:FirstOrderOptimality}) reads
\begin{align}
f_\eta = 2 \lambda_1 (w -q) + 2 \lambda_2 (w +q).
\label{eq:aux1}
\end{align}
We multiply the optimality condition by $q$ and obtain
\begin{alignat*}{3}
&\qquad &&f_\eta \cdot q &&= 2 \lambda_1 (w -q) \cdot q + 2 \lambda_2 (w +q) \cdot q \\
&\Leftrightarrow && f_\eta \cdot q &&= 2 (\lambda_2 - \lambda_1)~ |q|^2 \\
&\Leftrightarrow &&2 (\lambda_2 - \lambda_1) &&= f_\eta \cdot \dfrac{q}{|q|^2}.
\end{alignat*}
Multiplying (\ref{eq:aux1}) by $w$ yields
\begin{align}
f_\eta \cdot w = 2 (\lambda_1 + \lambda_2) |w|^2
\label{eq:aux2}
\end{align}
and another multiplication of (\ref{eq:aux1}) by $f_\eta$ yields
\begin{align*}
1 = |f_\eta|^2 &= 2 (\lambda_1 + \lambda_2) w \cdot f_\eta + 2 (\lambda_2 -\lambda_1) q\cdot f_\eta \\
&= 4 (\lambda_1 +\lambda_2)^2 |w|^2 + \left( f_\eta\cdot \dfrac{q}{|q|} \right)^2\\
&= (2 (\lambda_1 +\lambda_2) |w|)^2 + |\cos(\varphi)|^2,
\end{align*}
where we inserted the previous results in the last two steps. 
We rearrange and find
\begin{align*}
2 (\lambda_1 + \lambda_2) = \sqrt{ 1 - |\cos(\varphi)|^2} |w|^{-1} =  |\sin(\varphi)| |w|^{-1},
\end{align*}
which finally leads us to 
\begin{align*}
f_\eta \cdot w &= (\lambda_1 + \lambda_2) |w|^2 = |\sin(\varphi)| |w| = |\sin(\varphi)| \sqrt{1-|q|^2}.
\end{align*}
Hence we have
\begin{align*}
(f_1 \Box f_2)(\eta) \geq (f_1^* + f_2^*)^*(\eta) \geq \int_\Omega G(f_\eta,q) \di |\eta|.
\end{align*}
It remains to show that 
\begin{align*}
 (f_1 \Box f_2)(\eta) = \inf_{z \in \M} \| \eta - z \nM + \|z\nM - \langle q, \eta - 2z \rangle  \leq \int_\Omega G(f_\eta,q) \di |\eta|.
\end{align*}
At first we observe that for any $g \in L_{|\eta|}^1(\Omega, \R^m)$ we have that $z_* := g |\eta| \in \M$, since 
\begin{align*}
 \|z_*\nM = \int_\Omega |g| \di |\eta|.
\end{align*}
By the definition of the dual norm we furthermore obtain 
\begin{align*}
 &\quad \| \eta - z_*\nM + \|z_*\nM - \langle q, \eta - 2 z_* \rangle \\ 
 &= \sup_{\substack{\phi \in \CO \\ \|\phi\|_{\infty} \leq 1}} \int_\Omega \phi \cdot (f_\eta - g) \di |\eta| 
  + \sup_{\substack{\psi \in \CO \\ \|\psi\|_{\infty} \leq 1}} \int_\Omega \psi \cdot g \di |\eta|
  - \int_\Omega q \cdot (f_\eta - 2g) \di |\eta| \\
 &\leq \int_\Omega |f_\eta - g| + |g| - q \cdot (f_\eta - 2g) \di |\eta|.
\end{align*}
For $|\eta|$-a.e. $x \in \Omega$ we define 
\begin{align*}
 H(g(x)) := |f_\eta(x) - g(x)| + |g(x)| - q(x) \cdot (f_\eta(x) - 2g(x)),
\end{align*}
where we again leave out the dependence on $x$ in the following. 
We have to distinguish four cases. 

\noindent
\textbf{1st case:} If $|q| < \cos(\varphi)$, we have $q \cdot f_\eta > 0$ and we can obtain 
\begin{align*}
 H(g) = |f_\eta| - q \cdot f_\eta = |f_\eta| - |q \cdot f_\eta| = 1 - |\cos(\varphi)||q|
\end{align*}
for $g = 0$. 

\noindent
\textbf{2nd case:} Analogously if $|q| < - \cos(\varphi)$, we have $q \cdot f_\eta < 0$ and $g = f_\eta$ leads to 
\begin{align*}
 H(g) = |f_\eta| + q \cdot f_\eta = |f_\eta| - |q \cdot f_\eta| = 1 - |\cos(\varphi)||q|
\end{align*}

\noindent
\textbf{3rd case:}
In case $|q| = 1$ we let $c > 0$ and $g = f_\eta / 2 - (cq)/2$ to find
\begin{align*}
 H(g) &= \left| \frac{f_\eta}{2} + \frac{cq}{2} \right| + \left| \frac{f_\eta}{2} - \frac{cq}{2} \right| - c |q|^2 
 = \frac{c}{2} \left( \left| q + \frac{f_\eta}{c} \right| + \left| q - \frac{f_\eta}{c} \right| - 2 \right).  
\end{align*}
Using a Taylor expansion around $q$ and the fact that $|q| = 1$ we obtain
\begin{align*}
 \left| q + \frac{f_\eta}{c} \right| &= |q| + \frac{q}{|q|} \cdot \frac{f_\eta}{c} + O(c^{-2}) 
 = 1 + \frac{q}{|q|} \cdot \frac{f_\eta}{c} + O(c^{-2}), \\ 
 \left| q - \frac{f_\eta}{c} \right| &= |q| - \frac{q}{|q|} \cdot \frac{f_\eta}{c} + O(c^{-2}) 
 = 1 - \frac{q}{|q|} \cdot \frac{f_\eta}{c} + O(c^{-2}),
\end{align*}
which leads us to 
\begin{align*}
 H(g) = \frac{c}{2} ( 2 + O(c^{-2}) - 2) = O(c^{-1}) \to 0,
\end{align*}
for $c \to \infty$. 
Hence for every $\varepsilon > 0$ there exists a $c^* > 0$ such that $H(g) \leq \varepsilon / \| \eta \nM$.

\noindent
\textbf{4th case:} Finally, if $|q| < 1$ and $|q| \geq |\cos(\varphi)|$, we can pick $g = 2 \lambda_1 (w - q)$ 
with $\lambda_1$ and $w$ being the Lagrange multiplier and dual variable from the above computations of the third case.
We recall that $|w+q| = |w-q| = 1$, $w \cdot q = 0$, $|w| = \sqrt{1 - |q|^2}$ and $f_\eta - 2 \lambda_1 (w-q) = 2 \lambda_2 (w+q)$.
A straight forward calculation yields 
\begin{align*}
 H(g) &= |f_\eta - 2 \lambda_1 (w-q)| + |2 \lambda_1 (w-q)| - q \cdot ( f_ \eta - 2 \lambda_1 (w-q) - 2 \lambda_1 (w-q)) \\
 &= |2 \lambda_2 q (w+q) | + |2 \lambda_1 (w-q)| - q \cdot ( 2 \lambda_2 (w+q) - 2 \lambda_1 (w-q)) \\
 &= 2(\lambda_2 + \lambda_1) (1 - |q|^2) \\
 &= |\sin(\varphi)|\sqrt{1-|q|^2}.
\end{align*}

Hence for $|\eta|$-a.e. $x \in \Omega$ we define $g \colon \Omega \to \R^m$ as 
\begin{align*}
 g(x) := \begin{cases}
          0, & \text{ if } |q(x)| < \cos(\varphi(x)), \\
          f_\eta(x), & \text{ if } |q(x)| < - \cos(\varphi(x)), \\
          \frac{f_\eta(x)}{2} - \frac{c^*(x)}{2}q(x), & \text{ if } |q(x)| = 1, \\
          2 \lambda_1(x) (w(x) - q(x)), & \text{ if } |q(x)| > | \cos(\varphi(x))| \text{ and } |q(x)| < 1. 
         \end{cases}
\end{align*}
For every $\alpha > 0$ we let 
\begin{align*}
 g^\alpha(x):= \begin{cases}
                g(x), &\text{ if } |g(x)| \leq \alpha, \\
                0, &\text{ else.}
               \end{cases}
\end{align*}
$g^\alpha$ is bounded $|\eta|$-a.e., thus $ g^\alpha \in L_{|\eta|}^1(\Omega, \R^m)$ and $z_*^\alpha := g^\alpha |\eta| \in \M$.
We remark that since $G$ is bounded by $1,$ $|\eta|$-a.e., we have $\int_\Omega G(f_\eta,q) \di \eta \leq \|\eta\nM$, and compute 
\begin{align*}
 (f_1 \Box f_2)(\eta) 
 &\leq \| \eta - z_*^\alpha\nM + \|z_*^\alpha\nM - \langle q, \eta - 2 z_*^\alpha \rangle \\
 &\leq \int_\Omega H(g^\alpha) \di |\eta| \\
 &= \int_{ \{ |g| \leq \alpha \} } G(f_\eta,q) + \frac{\varepsilon}{\|\eta\nM} \di |\eta| + \int_{ \{ |g| > \alpha \} } 1 - q \cdot f_\eta \di |\eta| \\
 &= \int_\Omega G(f_\eta,q) \di |\eta| + \int_{ \{ |g| \leq \alpha \} } \frac{\varepsilon}{\|\eta\nM} \di |\eta| + \int_{ \{ |g| > \alpha \} } 1 - q \cdot f_\eta - G( f_\eta,q) \di |\eta| \\
 &\leq \int_\Omega G(f_\eta,q) \di |\eta| + \int_{ \{ |g| \leq \alpha \} } \frac{\varepsilon}{\|\eta\nM} \di |\eta| + \int_{ \{ |g| > \alpha \} } 3 \di |\eta| \\
 & \to \int_\Omega G(f_\eta,q) \di |\eta| + \varepsilon
\end{align*}
for $\alpha \to \infty$. 
Since $\varepsilon$ was arbitrary, this completes the proof.
\end{proof}

\subsection{Numerical solution of PET-MRI joint reconstruction}
For the solution of the joint PET-MRI reconstruction scheme \eqref{eq:PETMRIJointReconstruction1}, \eqref{eq:PETMRIJointReconstruction2} we only need to further specify the convex conjugates of the involved data terms and the role of the sets $C_i$.
We use the notation from Section \ref{sec:petmri} in order to be consistent.
Since we constrain both the PET and the MR image to be real-valued and positive, we let $C = \R_0^+$.
Hence the full problem reads:
\begin{alignat*}{3}
u^{k+1} &\in \arg \min_u ~ \alpha \DKL(f,Au) &&+ \lambda D_{\TV}^{p_u^k}(u,u^k) && + (1-\lambda) \ICB^{p_v^k}(u,v^k) + \delta_C(u), \\
v^{k+1} &\in \arg \min_v ~ \dfrac{\beta}{2} \| Bv - g \|_2^2 &&+ \mu D_{\TV}^{p_v^k}(v,v^k) &&+ (1-\mu) \ICB^{p_u^k}(v,u^k) + \delta_C(v).	
\end{alignat*}
Since both data terms are differentiable, their convex conjugates are easy to compute. 
The convex conjugate of the scaled Kullback-Leibler divergence 
\begin{align*}
	H_{f,\alpha}(w) = \alpha \DKL (f,w) = \alpha \sum_{i=1}^M w_i - f_i \log w_i.
\end{align*} 
is given by 
\begin{align*}
	H_{f,\alpha}^*(y) = \sum_{i=1}^{M} \alpha f_i \Big[ \log \left( \frac{\alpha f_i}{\alpha-y_i}  \right) -1 \Big].
\end{align*}
The associated proximal map is 
\begin{align*}
	\prox_{\sigma H_{f,\alpha}^*}(y) = \dfrac{\alpha + y}{2} - \sqrt{\dfrac{(y - \alpha )^2}{4} + \sigma \alpha f}.
\end{align*}
For the quadratic MR data term 
\begin{align*}
	H_{g,\beta}(w) = \dfrac{\beta}{2} \| w - g \|_{\C}^2
\end{align*}
we derive the conjugate 
\begin{align*}
	H_{g,\beta}^*(y) = \dfrac{1}{2\beta} \|y\|_{\C}^2 + \Real \langle y,g \rangle 
\end{align*}
and the proximal map 
\begin{align*}
	\prox_{\sigma H_{g,\beta}^*}(y) = \dfrac{\beta y - \sigma \beta g}{\beta + \sigma}.
\end{align*} 
The primal proximal map for $\delta_{C}$ can be carried out as a simple projection onto the nonnegative real numbers. 
Regarding step sizes, we employ the diagonal preconditioning presented in \cite{Pock:DiagonalPreconditioning}, using matrix representations of our operators to precompute $\sigma$ and $\tau$.
\end{appendix}